\documentclass[11pt]{article}
\usepackage{algorithmic}
\usepackage{appendix}

\usepackage[a4paper, margin={.15\paperwidth,.08\paperheight},
marginratio=1:1]{geometry}

\usepackage{graphicx, color}
\usepackage{authblk}

\definecolor{darkblue}{rgb}{0.0,0.0,0.7}

\usepackage[colorinlistoftodos]{todonotes}
 
\usepackage{amsmath}
\usepackage{amsthm}
\usepackage{amsfonts}
\usepackage{latexsym}
\usepackage{multicol}
\usepackage{fullpage}
\usepackage{amssymb}
\usepackage{mathrsfs}
\usepackage{color}
\usepackage{algorithmic, algorithm}
\newtheorem*{rep@theorem}{\rep@title}
\newcommand{\newreptheorem}[2]{%
	\newenvironment{rep#1}[1]{%
		\def\rep@title{#2 \ref{##1}'}%
		\begin{rep@theorem}}%
		{\end{rep@theorem}}}
\makeatother

\newtheorem{theorem}{Theorem}
\newreptheorem{theorem}{Theorem}

\newtheorem{remark}{Remark}
\newtheorem{corollary}{Corollary}
\newtheorem{proposition}{Proposition}
\newtheorem{lemma}{Lemma}

\usepackage{chngcntr}
\usepackage{bbm}
\usepackage{makeidx}

\usepackage{dsfont}

\newtheorem{assumption}{Assumption}

\usepackage{mathrsfs}

\usepackage{algorithmic, algorithm}

\definecolor{darkblue}{rgb}{0.0,0.0,0.7}

\newcommand{\norm}[1]{\|#1\|}%
\newcommand{\con}{\mathrm{con}}
\newcommand{\normop}[1]{\norm{#1}_{\op}}

\newcommand{\Bignormop}[1]{\Big\|#1\Big\|_{\op}}

\newcommand{\dsone}{{\mathds 1}}

\newcommand{\tT}{\mathds{T}}

\newcommand{\ind}[1]{\mathbf 1_{#1}}%

\newcommand{\E}{\mathbbm E}

\newcommand{\cS}{\mathscr S}

\renewcommand{\P}{\mathbbm P}

\newcommand{\sS}{\mathscr S}

\newcommand{\cadlag}{c\`adl\`ag}

\newcommand{\bA}{{\boldsymbol A}}%
\newcommand{\bB}{{\boldsymbol B}}%
\newcommand{\bC}{{\boldsymbol C}}%
\newcommand{\bD}{{\boldsymbol D}}%
\newcommand{\bE}{{\boldsymbol E}}%
\newcommand{\bG}{{\boldsymbol G}}%
\newcommand{\bH}{{\boldsymbol H}}%
\newcommand{\bI}{{\boldsymbol I}}%
\newcommand{\bJ}{{\boldsymbol J}}%
\newcommand{\bM}{{\boldsymbol M}}%
\newcommand{\bN}{{\boldsymbol N}}%
\newcommand{\bO}{{\boldsymbol 0}}%
\newcommand{\bP}{{\boldsymbol P}}%
\newcommand{\bR}{{\boldsymbol R}}%
\newcommand{\bS}{{\boldsymbol S}}%
\newcommand{\bU}{{\boldsymbol U}}%
\newcommand{\bV}{{\boldsymbol V}}%
\newcommand{\bW}{{\boldsymbol W}}%
\newcommand{\bX}{{\boldsymbol X}}%
\newcommand{\bY}{{\boldsymbol Y}}%
\newcommand{\bZ}{{\boldsymbol Z}}%
\newcommand{\blambda}{{\boldsymbol \lambda}}%
\newcommand{\bLambda}{{\boldsymbol \Lambda}}%

\newcommand{\lmax}{\lambda_{\max}}

\newcommand{\bul}{\bullet}

\DeclareMathOperator{\op}{op}
\DeclareMathOperator{\tr}{tr}

\DeclareMathOperator{\diag}{diag}
\renewcommand{\vec}{\mathrm{vec}}

\newcommand{\1}{{\rm 1}\kern-0.24em{\rm I}}

\newcommand{\bs}{\boldsymbol}

\newcommand{\cF}{\mathscr F}

\newcommand{\R}{\mathbb R}
\newcommand{\N}{\mathbb N}

\newcommand{\inr}[1]{\langle #1 \rangle}

\newcommand{\Jmax}{J_{\max}}

\newcommand{\mgeq}{\succcurlyeq}%
\newcommand{\mleq}{\preccurlyeq}

\begin{document}

\title{Concentration inequalities for matrix martingales \\ in continuous time}

\author[1]{Emmanuel Bacry}
\author[1]{St\'ephane Ga\"iffas}
\author[1,2]{Jean-Fran\c{c}ois Muzy}
\affil[1]{\small Centre de Mathe\'ematiques Appliqu\'ees, \'Ecole Polytechnique and CNRS \authorcr UMR 7641, 91128 Palaiseau, France}
\affil[2]{\small Laboratoire Sciences Pour l'Environnement, CNRS, Universit\'e de Corse, \authorcr UMR 6134, 20250 Cort\'e, France}

\maketitle

\begin{abstract}
  This paper gives new concentration inequalities for the spectral
  norm of a wide class of matrix martingales in continuous time. 
  These results extend previously established Freedman and Bernstein inequalities for series of random matrices to the 
  class of continuous time processes.
  Our analysis relies on a new supermartingale property of the trace
  exponential proved within the framework of stochastic calculus. 
  We provide also several examples that illustrate the fact that our results allow us to recover easily several formerly obtained sharp bounds for discrete time matrix martingales.
\end{abstract}

\section{Introduction}

Matrix concentration inequalities control the deviation of a random matrix around its mean. Until now, results in literature consider the case of sums of independent random matrices, or matrix martingales in discrete time. A first matrix version of the Chernoff bound is given in~\cite{ahlswede2002strong} and was adapted to yield matrix analogues of standard scalar concentration inequalities in~\cite{christofides2008expansion, oliveira2009concentration, oliveira2010sums, minsker2011some}. Later, these results were improved in~\cite{tropp2012user, tropp2011freedman}, by the use of a theorem due to Lieb~\cite{lieb1973convex}, about the concavity of a trace exponential function, which is a deep result closely related to the joint convexity of quantum entropy in physics. See also~\cite{mackey2014matrix} for a family of sharper and more general results based on the Stein's method. These works contain extensions to random matrices of classical concentration inequalities for sums of independent scalar random variables, such as the Bernstein inequality for sub-exponential random variables, Hoeffding inequality for sub-Gaussian random variables or Freedman inequality for martingales, see e.g.~\cite{massart2007concentration} for a description of these classical inequalities. 
Matrix concentration inequalities have a plethora of applications, in particular in compressed sensing and statistical estimation~\cite{koltchinskii2011oracle}, to develop a simpler analysis of matrix completion~\cite{Gross11,recht2011simpler}, for matrix regression~\cite{negahban2012restricted, Koltchinskii11,rohde2011estimation, bunea2011optimal}, for randomized linear algebra~\cite{gittens2011spectral, mackey2011divide}, and robust PCA~\cite{Candes09b}, which are some examples from a large corpus of works.
On the other hand, concentration inequalities for scalar continuous-time martingales are well-known, see for instance \cite{liptser1989theory,van1995exponential} and \cite{reynaud2006compensator} for uniform versions of scalar concentration, and have a large number of applications in high-dimensional statistics \cite{hansen_reynaud_bouret_viroirard, gaiffas2012high} among many others.

Matrix martingales in continuous time are probabilistic objects that
naturally appear in many problems like e.g., for statistical learning of time-dependent
systems. No extension of the previously mentioned results in this framework is
available in literature. The aim of this paper is to provide such
results, by combining tools from random matrix theory and from stochastic calculus~\cite{liptser1989theory}. We establish concentration inequalities for a large class of continuous-time matrix martingales with
arbitrary predictable quadratic covariation tensor.
More precisely, we provide a matrix version of Freedman's inequality for purely discontinuous matrix martingales (see Theorem~\ref{thm:concentration_counting}) as well as continuous martingales (see
Theorem~\ref{thm:concentration_continuous}). 
We show that the variance term in the concentration inequalities is provided by the largest eigenvalue of the predictable quadratic covariation matrix.
In that respect, our results can be understood as extensions to continuous time martingales of previously established results by Tropp~\cite{tropp2011freedman} in the case of discrete time matrix martingales.
Our proofs techniques required a very different analysis than the discrete-time case, involving tools from stochastic calculus.

%Matrix martingales in
%continuous time are probabilistic objects that appear naturally in
%several problems, in particular in models that describe time-dependent
%systems, such as models used for the study of social networks at a
%``microscopic'' time scale. The microscopic time scale stands for a recent approach for the study of social-networks: instead of studying a fixed macroscopic graph of
%connected nodes (users, products) that are linked by edges (click,
%like, friendship, retweet, etc.), it considers instead
%the network at a more microscopic scale, by looking at the timestamps of the users actions, in order to recover an implicit connectivity of users. This approach has known a strong development recently: cascade models based on survival analysis are developped in~\cite{rodriguez2011uncovering,gomez13,gomez14a}, and models based on self-exciting point processes, namely
%the Hawkes process~\cite{hawkes1971spectra} are given in~\cite{crane2008robust, blundell2012modelling, zhou_zha_le_2013, yang2013mixture,linderman2014discovering,dubois2013stochastic,iwata2013discovering}, among others.

The paper is organized as follows. In
Section~\ref{sec:main_results}, after introducing some notations and defining the class of matrix martingales we consider, we state our main results for purely discontinuous martingales, see Section~\ref{sec:discountinuous_martingale} and continuous matrix martingales, see Section~\ref{sec:continuous_martingale}.
We provide some comments about the link with analogous results in discrete time and the sharpness of our bounds. In Section~\ref{sec:Applications}, we discuss some examples of application of our Theorems. We consider various particular situations that notably allows us to 
recover some known results concerning series of random matrices or scalar point processes.
% We also give a quick illustration of a use of our new concentration results in the context of low-rank multivariate Hawkes processes estimation, which is a family of self-excited counting processes of growing interest in machine learning, in particular in the field of finance and social networks, see for instance~\cite{bacry2013modelling} and~\cite{crane2008robust}.
Technical results and the proofs of the Theorems are gathered in Appendices~\ref{sec:tools},~\ref{sec:proof_of_theorem_1} and~\ref{sec:proof_theorem2}.
We notably establish a matrix supermartingale property (Proposition~\ref{prop:new-counting-supermart}) that is essential for obtaining the concentration inequalities.
%Our new concentration results allows for the theoretical study of such models, in particular for the multivariate Hawkes model, since it gives a sharp control of the noise term, which is in this case a purely discountinuous matrix martingale (see Section~\ref{sec:discountinuous_martingale}), as it comes from the compensation of a counting process.
%We give a quick illustration (see~Section~\ref{sec:Applications}) of a use of our new concentration results to low-rank multivariate Hawkes processes: low-rank modeling is now a standard approach in collaborative filtering problems~\cite{koren2010collaborative}, where a very popular strategy is to use a penalization technique based on a convex relaxation of the rank, given by the trace normn (see~\cite{Candes04, Candes09} and references mentioned above), which is given by the sum of singular values.
%For the sake of clarity, some technical proofs are gathered in Appendix \ref{appendix:prop3}.

\section{Main results}
\label{sec:main_results}

In this section, we give the main results of the paper, namely Theorems~\ref{thm:concentration_counting} and~\ref{thm:concentration_continuous}, 
that provide concentration inequalities for purely discontinuous and continuous matrix martingales.
We first begin by recalling some definitions from probability theory and stochastic calculus and 
set some notations.

\subsection{Probabilistic background and Notations}
\label{sec:probback}
% \subsection{Probabilistic background}

\paragraph{Probabilistic background.} % (fold)
\label{par:proba}

% paragraph proba (end)

We consider a complete probability space $(\Omega, \cF, \P)$ and a
filtration $\{ \cF_t \}_{t \geq 0}$ of $\sigma$-algebras included in
$\cF$. Expectation $\E$ is always taken with respect to $\P$. We
assume that the stochastic system $(\Omega, \cF, \{ \cF_t \}_{t \geq 0}, \P)$ satisfies the \emph{usual} conditions, namely that $\cF_0$ is augmented by the $\P$-null sets, and that the filtration is right continuous, namely $\cF_t = \cap_{u > t} \cF_u$ for any $t \geq 0$. We shall denote $\cF_{t^-}$ as the smallest $\sigma$-algebra containing all $\cF_s$ for $s < t$.

A matrix-valued stochastic processes $\{ \bX_t \}_{t \geq 0}$ is a
family of random matrices of constant size (e.g. $m \times n$) defined on $(\Omega, \cF, \P)$. We say that $\{ \bX_t \}_{t \geq 0}$ is adapted if for each $t \geq 0$, all the entries of $\bX_t$ are $\cF_t$-measurable. We say that it is \cadlag~if the trajectories on $[0, +\infty]$ of each entries have left limits and
are right continuous for all $\omega \in \Omega$. If $\{ \bX_t \}_{t \geq 0}$
is \cadlag, then we define its jump process $\{ \Delta \bX_t \}_{t \geq 0}$ where $\Delta \bX_t = \bX_t - \bX_{t^-}$. We say that $\{ \bX_t \}_{t \geq 0}$ is predictable if all its entries are \cadlag~and predictable.
We recall that a {\em predictable process} is a process that is measurable with respect to the $\sigma$-field
generated by left-continuous adapted processes. In particular, if $\tau$ is a stopping time and $X_t$
is a predictable process, then $X_\tau$ is $\cF_{\tau^-}$ measurable.

A \emph{matrix semimartingale} is a matrix-valued stochastic process whose entries are all semimartingales.
In the same way, a \emph{matrix martingale} $\{ \bM_t \}_{t \geq 0}$ is a matrix-valued stochastic process with entries that are all
martingales. Namely, we assume that for all possible indexes $(i, j)$ of entries, $(\bM_t)_{i, j}$ is adapted, \cadlag, such that $\E |(\bM_t)_{i, j}| < +\infty$ for all $t \geq 0$, and that
\begin{equation*}
  \E[\bM_t | \cF_s] = \bM_s
\end{equation*}
for any $0 \leq s \leq t$, where the conditional expectation is applied
entry-wise on $\bM_t$. More generally, expectations and conditional
expectations are always applied entry-wise. A brief review of tools from stochastic calculus
based on semimartingales is provided in appendix \ref{sec:ito}.

\paragraph{Notations.}
\label{sec:notations}

We denote by $\dsone$ the column vector with all entries equal to $1$ (with size depending on the context).
Let $\bX$ be a real matrix and $x$ a real vector.
The notations $\diag[x]$ stands for the diagonal matrix with diagonal equal to $x$, while if $\bX$ is a square matrix, $\diag[\bX]$ stands for the diagonal matrix with diagonal equal to the one of $\bX$, $\tr \bX$ stands for the trace of $\bX$. The operator norm (largest singular value) will be denoted by
$\normop{\bX}$.
We define also $|\bX|$ by taking the absolute value of each entry of $\bX$.

If $\bY$ is another real matrix,
the notation $\bX \odot \bY$ stands for the entry-wise product (Hadamard product) of $\bX$ and $\bY$ with same dimensions, namely $(\bX \odot \bY)_{j, k} = (\bX)_{j, k} (\bY)_{j, k}$. We shall denote by $\bX^{\odot k}$ the Hadamard power, where each entry of $\bX^{\odot k}$ is the $k$-th power of the corresponding entry of $\bX$.

We also denote $\bX_{\bul, j}$ for the $j$-th column of $\bX$ while $\bX_{j, \bul}$ stands for the $j$-th row.
Moreover, for a matrix $\bX$ and $p \geq 1$, we define the norms \begin{equation*}
  \norm{\bX}_{p, \infty} = \max_j \norm{\bX_{j, \bul}}_p \; \text{
    and } \; \norm{\bX}_{\infty, p} = \max_j \norm{\bX_{\bul, j}}_p,
\end{equation*}
where $\norm{\cdot}_p$ is the vector $\ell_p$-norm.

For a symmetric matrix $\bX$, the largest
eigenvalue  is denoted $\lmax(\bX)$.
Moreover, the symbol $\mleq$ stands for the positive 
semidefinite (p.s.d.) order on symmetric matrices, namely 
$\bs X \mleq \bs Y$ iff $\bs Y - \bs X$ is p.s.d.

We shall denote, when well-defined, $\int_0^t \bX_s ds$ for the
matrix of integrated entries of $\bX_s$, namely $(\int_0^t \bX_s
ds)_{i, j} = \int_0^t (\bX_s)_{i, j} ds$.
% , and note also that, when it
% makes sense, we have
% \begin{equation*}
%   \bA_s \mleq \bB_s \quad \forall s \geq 0 \; \Rightarrow \; \int_0^t
%   \bA_s ds \mleq \int_0^t \bB_s ds.
% \end{equation*}
We use also matrix notations for stochastic integrals, for instance $\int_0^t \bX_s d \bY_s$ stands for the matrix
with entries given by the stochastic integral 
\begin{equation*}
  \Big(\int_0^t \bX_s d \bY_s\Big)_{i, j} = \sum_k \int_0^t
  (\bX_s)_{i, k} d (\bY_s)_{k, j},
\end{equation*}
for stochastic processes $\bX_t$ and $\bY_t$ that are matrix-valued, such that the matrix product $\bX_t \bY_t$ makes sense, and such that these stochastic integrals are well-defined for all $i, j$.
We define similarly $\int_0^t d\bX_s \bY_s$.

Let $\tT$ be a rank 4 tensor of dimension $(m \times n \times p \times q)$ . It can be considered as a linear mapping from
$\R^{p \times q}$ to $\R^{m \times n}$ according to the following ``tensor-matrix'' product:
\begin{equation*}
  (\tT \circ \bA )_{i,j} = \sum_{k=1}^p \sum_{l=1}^q \tT_{i,j;k,l} 
  \bA_{k,l}.
\end{equation*}
We will denote by $\tT^{\top}$ the tensor such that $\tT^\top \circ \bA = (\tT \circ \bA)^{\top}$ (i.e., $\tT^{\top}_{i,j;k,l} = \tT_{j,i;k,l}$)
and by $\tT_{\bul;k,l}$ and $\tT_{i,j;\bul}$ the matrices obtained when fixing respectively the indices $k,l$ and $i,j$.
Notice that $(\tT \circ \bA )_{i,j} = \tr (\tT_{i,j;\bul} \bA^{\top})$.
If $\tT$ et $\tT'$ are two tensors of dimensions respectively $m \times n \times p \times q$ and $n \times r \times p \times q$, $\tT \tT'$ will stand for the tensor of dimension $m \times r \times p \times q$ defined as $(\tT \tT')_{i,j;k,l} = (\tT_{\bul;k,l} \tT'_{\bul;k,l})_{i,j}$. 
Accordingly, for an integer $r \geq 1$, if $\tT_{\bul;a,b}$ are square matrices, we will denote by $\tT^{r}$ the tensor such that 
$(\tT^{r})_{i,j;k,l} = (\tT_{\bul;k,l}^r)_{i,j}$.
We also introduce $\norm{\tT}_{\op; \infty} = \max_{k, l} 
\normop{\tT_{\bul; k, l}}$, the maximum operator norm of all matrices formed by the first two dimensions of tensor $\tT$.

\paragraph{The matrix martingale $\bZ_t$.}

In this paper we shall consider the class of $m \times n$ 
matrix martingales that can be written as
\begin{equation}
  \label{eq:martingale}
   \bZ_t = \int_0^t \tT_s \circ (\bC_s \odot d \bM_s),
\end{equation}
where $\tT_s$ is a rank 4 tensor with dimensions $m \times n \times p \times q$, whose components are assumed to be locally bounded predictable random functions. 
The process $\bM_t$ is a $p \times q$ is matrix with entries that are square integrable martingales with a diagonal quadratic covariation matrix (see Section~\ref{sec:ito} for the definition of the quadratic covariation matrix of a semimartingale matrix).
The matrix $\bC_s$ is a matrix of $p \times q$ predictable locally bounded functions.

More explicitly, the entries of $\bZ_t$ are given by
\begin{equation*}
(\bZ_t)_{i, j} = \sum_{k=1}^p \sum_{l=1}^q  \int_0^t (\tT_s)_{i,j;k,l} 
(\bC_s)_{k,l} (d \bM_s)_{k,l}.
\end{equation*}

Note that Equation~\eqref{eq:martingale} corresponds to a wide class of matrix martingales. 
This shape of matrix martingale, which involves a rank-4 tensor, is natural: all quadratic covariations between pairs of entries of $\bZ_t$ are accounted by the linear transformation $\tT$.
Let us remark that if one chooses $\tT_{i,j;k,l} = (\bA_s)_{i,k} (\bB_s)_{l,j}$ where 
$\bA_s$ and $\bB_s$ are respectively $m \times p$ and $q \times n$ matrices of predictable functions, then
\begin{equation}
\label{eq:martingale_old}
% \label{eq:martingale-form}
\bZ_t = \int_0^t \bA_s (\bC_s \odot d \bM_s) \bB_s.
\end{equation}
If one chooses $\tT$ of dimensions $(m \times n \times 1 \times 1)$ and $\bM_t = M_t$ a scalar martingale, then
\begin{equation*}
  \bZ_t = \int_0^t \bA_s dM_s,
\end{equation*}
where $(\bA_s)_{i,j} = (\tT_s)_{i,j; 1,1}$ is a constant matrix linear transform.

In Section~\ref{sec:Applications} below, we prove that such particular cases lead to generalizations to continuous-time martingales of previously known concentration inequalities for ``static'' random matrices.
In the following we will distinguish situations where the entries of 
$\bM_t$ are purely discontinuous martingales (see Section~\ref{sec:discountinuous_martingale})
and continuous martingales (see Section~\ref{sec:continuous_martingale}).
We recall the definitions of continuous and purely discontinuous martingales, along with other important notions from stochastic calculus in Section~\ref{sec:ito}.

% If $\bX$ and $\bY$ are real matrices with the same shape, we denote
% their inner product by
% \begin{equation*}
%   \inr{\bX, \bY} = \tr(\bX^\top \bY) = \sum_j \sum_k (\bX)_{j, k}
%   (\bY)_{j, k}.
% \end{equation*}
% The associated norm is the Frobenius norm given by $\norm{\bX}_F =
% \sqrt{\inr{\bX, \bX}}$.

% Notice that one has:
% \begin{equation}
% \label{vec2}
% \vec \bX =
%   \begin{bmatrix}
%     \bX_{\bul, 1}^\top \bX_{\bul, 2}^\top \cdots \bX_{\bul, q}^\top
%   \end{bmatrix}^\top.
% \end{equation}

\subsection{Purely discontinuous matrix martingales}
\label{sec:discountinuous_martingale}

In this section, we consider the case of a purely discontinuous
(this notion is defined in Appendix~\ref{sec:ito}) martingale $\bM_t$. 
More specifically, we assume that $\bM_t$ is a martingale coming from the compensation of a random matrix with entries that are  compound counting processes. We denote by $[\vec \bM]_t$ the quadratic covariation matrix of the vectorization of $\bM_t$ (defined in Appendix~\ref{sec:ito}).
\begin{assumption}
  \label{ass:M_d}
  Assume that $\bM_t$ is a
  purely discontinuous matrix-martingale with entries that are locally bounded. Moreover, we assume that they 
  do not jump at the same time, i.e. $[\vec \bM]_t$ is a diagonal matrix for any $t$.
  Moreover, assume that any $(i, j)$-entry satisfies
  \begin{equation}
  \label{eq:defMdis}
  (\Delta \bM_t)_{i, j} = (\bJ_{(\bN_t)_{i, j}})_{i, j} \times (\Delta \bN_t)_{i, j}
  \end{equation}
  where\textup:
  \begin{itemize}
  \item $\bN_t$ is a $p \times q$ matrix counting process \textup(i.e., each component is a counting process\textup) with an intensity process $\blambda_t$ which is predictable, continuous and with finite variations \textup(FV\textup)\textup;
  \item $(\bJ_n)_{n \in \N}$ is a sequence of $p \times q$ random matrices, independent of $(\tT_t)_{t \geq 0}, (\bC_t)_{t \geq 0}$ and $(\bN_t)_{t \geq 0}$ and identically distributed, such that $|(\bJ_1)_{i, j}| \leq \Jmax$ a.s. for any $i, j$ and $k \geq 1$, 
  where $\Jmax > 0$.
\end{itemize}
\end{assumption}

\begin{remark}
  Equation~\eqref{eq:defMdis} can be rewritten for short as 
$\Delta \bM_t = \bJ_{\bN_t} \odot \Delta \bN_t$. 
It imposes a mild condition on the structure of the jumps of $\bM_t$ that allows to derive an explicit form for the compensator of
 $\bZ_t$ \textup(see Section~\ref{sec:proof_of_theorem_1}\textup). 
Note that if $\bM_t$ is the martingale associated with a matrix counting process, then one can simply choose the sequence
$(\bJ_n)_{n \in \N}$ as constantly equal to the matrix filled with ones.
\end{remark}

The next Theorem is a concentration inequality for
$\normop{\bZ_t}$, the operator norm of $\bZ_t$.
Let $\inr{\bZ_{\bul,j}}_t$ (resp. $\inr{\bZ_{j,\bul}}_t$) be 
the matrices of predictable quadratic variations of the column (resp. row) vector $(\bZ_t)_{\bul,j}$  (resp. $(\bZ_t)_{\bul,j}$), and 
let us define
\begin{equation}
  \label{eq:def_sigma_Z_t}
  \sigma^2(\bZ_t) = \max \bigg( \Big\| \sum_{j=1}^n \inr{\bZ_{\bul,j}}_t 
  \Big\|_{\op}, \Big\| \sum_{j=1}^m \inr{\bZ_{j, \bul}}_t \Big\|_{\op} \bigg).
\end{equation}
Let us introduce also
\begin{equation}
  \label{eq:W_def}
  \bW_s = 
  \begin{bmatrix}
   \tT_s \tT_s^\top \circ  \left(\E(\bJ_1^{\odot 2}) \odot \bC^{\odot 2}_s \odot \blambda_s \right) & \bO \\ \bO &
\tT_s^\top \tT_s \circ \left( \E(\bJ_1^{\odot 2}) \odot \bC^{\odot 2}_s \odot \blambda_s \right) 
  \end{bmatrix},
\end{equation}
and
\begin{equation}
  \label{eq:b_t_def}
  b_t = J_{\max} \sup_{s \in [0, t]} \norm{\bC_s}_\infty 
 \max\big( \norm{\tT_s}_{\op; \infty},\norm{\tT^\top}_{\op; \infty}
 \big),
\end{equation}
and finally $\phi(x) = e^x - 1 - x$ for $x \in \R$.
\begin{theorem}
  \label{thm:concentration_counting}
  Let $\bZ_t$ be the $m \times n$ matrix martingale given by Equation~\eqref{eq:martingale} and suppose that Assumption~\ref{ass:M_d} holds. 
  Moreover, assume that
    \begin{equation}
    \label{eq:ass_thm1}
    \E \bigg[ \int_0^t
    \frac{\phi\big(3 J_{\max} \norm{\bC_s}_\infty 
    \max(\norm{\tT_s}_{\op; \infty},\norm{\tT_s^\top}_{\op; \infty})\big)}{J_{\max}^2 
    \norm{\bC_s}_\infty^2  \max(\norm{\tT_s}^2_{\op; \infty},\norm{\tT_s^\top}^2_{\op; \infty})} 
    (\bW_s)_{i, j} ds \bigg] < +\infty,
    \end{equation}
for any $1 \leq i, j \leq m + n$.
  Then, for any $t, x, b, v > 0$, the following holds\textup:
  \begin{equation*}
    \P \bigg[ \normop{\bZ_t} \geq \sqrt{2 v (x + \log (m + n))} +
    \frac{b (x + \log (m + n))}{3}, \quad \sigma^2(\bZ_t) \leq v,
    \quad b_t \leq b \bigg] \leq e^{-x},
  \end{equation*}
  where $\sigma^2(\bZ_t)$ is given by Equation~\eqref{eq:def_sigma_Z_t} and $b_t$ by Equation~\eqref{eq:b_t_def}.
  Moreover, we have
  \begin{equation*}
    \sigma^2(\bZ_t) = \lmax(\bV_t),
  \end{equation*}
  where
  \begin{equation}
    \label{eq:def_V_t_counting}
    \bV_t = \int_0^t \bW_s \; ds,
  \end{equation}
  with $\bW_s$ given by Equation~\eqref{eq:W_def}.
\end{theorem}

This theorem is proved in Appendix \ref{sec:proof_of_theorem_1}. It provides a first non-commutative version of a concentration inequality for continuous time matrix martingales, in the purely discontinuous case. 
This theorem can be understood as the generalization to continuous time martingales of a Freedman inequality for (discrete time) matrix martingales established in~\cite{tropp2011freedman}.

% \begin{reptheorem}{thm:concentration_counting}
% \label{th:theorem1prime}
% 	Let $\bZ_t$ be the $m \times n$ martingale considered in 
%   Theorem~\ref{thm:concentration_counting}. 
% 	Let $b_t$ as defined in Eq. \eqref{eq:b_t_def} and let
%    $\inr{\bZ_{\bul,j}}_t$ (resp. $\inr{\bZ_{j,\bul}}_t$) be 
% the matrices of predictable quadratic variations of the column (resp. row) vector $(\bZ_t)_{\bul,j}$  (resp. $(\bZ_t)_{\bul,j}$) . Let us define
% 	\begin{equation}
%     \sigma^2(\bZ_t) = \max \left( \normop{\sum_{j=1}^n \inr{\bZ_{\bul,j}}_t}, \normop{\sum_{j=1}^m \inr{\bZ_{j,\bul}}_t} \right)
% 	\end{equation}
% Then we have, for any $t,z,b,v >0$,
% \begin{equation}
%  \P \bigg[ \normop{\bZ_t} \geq z, \; b_t \leq b \; \; \mathbf{and} \; \; \sigma^2(\bZ_t) \leq v \bigg] \leq (m+n) e^{-\frac{z^2}{2 v+z b /3}}
% \end{equation}

% \end{reptheorem}

Let us notice that the two terms involved in the definition~\eqref{eq:def_sigma_Z_t} of $\sigma^2(\bZ_t)$ are precisely the matrices of predictable quadratic variations of the entries of respectively $\bZ_t \bZ_t^\top$ and $\bZ_t^\top \bZ_t$ in full agreement 
with the form provided in the discrete 
case~\cite{tropp2011freedman}.
Moreover, if $\tT_s$, $\blambda_s$ and $\bC_s$ are deterministic, we can actually write
\begin{equation}
  \label{eq:def_sigma_Z_t_deter}
\sigma^2 (\bZ_t) = \max \Big( \normop{\E(\bZ_t \bZ_t^\top)}, \normop{ \E(\bZ_t^\top \bZ_t)} \Big).
\end{equation}
This term has the same shape as the variance term from Bernstein 
inequality established for random series of bounded matrices $\bZ_n = \sum_k \bS_k$ as e.g., in~\cite{tropp2012user}.
This illustrates the fact that Theorem~\ref{thm:concentration_counting} extends former results for discrete series of random matrices to continuous time matrix martingales.
A detailed discussion and comparison with literature is given in Section~\ref{sec:Applications} below.

Note that, since $\phi$ is an increasing function,~\eqref{eq:ass_thm1} is satisfied whenever $\bV_t$ has finite expectation and both $\norm{\bC_s}_\infty$, $\norm{\tT_s}_{\op; \infty}$ and $\norm{\tT_s^\top}_{\op; \infty}$ are bounded a.s. by some fixed constant.
In the scalar case ($m = n = p = q = 1$), the assumption required in Equation~\eqref{eq:ass_thm1} becomes
\begin{equation*}
  \E \Big[ \int_0^t e^{3 |C_s|} \lambda_s ds \Big] < +\infty,
\end{equation*}
where matrices $\bA_t$, $\bB_t$ and tensor $\tT_t$ are scalars equal to one, and $\bC_t = C_t$ is scalar. 
This matches the standard assumption for an exponential deviation of the scalar martingale $Z_t = \bZ_t$, see for instance~\cite{bremaud1981point}.

\subsection{Concentration inequality for continuous matrix martingales}
\label{sec:continuous_martingale}

In this section, we study the matrix-martingale $\bZ_t$ given by~\eqref{eq:martingale} when it is continuous. 
This mainly amounts to consider situations where the compensated counting processes are replaced by Brownian motions. 
More specifically, we will suppose that $\bM_t$ satisfies the following.
\begin{assumption}
  \label{ass:M_c}
  Assume that $\{ \bM_t\}$ is a matrix of independent standard Brownian motions.
  This notably implies that its entry-wise predictable quadratic variation matrix reads 
  \begin{equation*}
    \inr{\bM}_t  =  t \bI.
  \end{equation*}
\end{assumption}

In this context, we can prove the analog of Theorem~\ref{thm:concentration_counting}, i.e, a Freedman concentration inequality for $\normop{\bZ_t}$, the operator norm of $\bZ_t$.
Thus, following the same lines as in the previous section, let $\sigma^2(\bZ_t)$ be defined by Equation~\eqref{eq:def_sigma_Z_t} and let us consider for following matrix:
\begin{equation}
  \label{eq:W_def_cont}
  \bW_t = \begin{bmatrix}
    \tT_t \tT_t^\top \circ \bC^{\odot 2}_t   & \bO \\ \bO &
    \tT_t^\top \tT_t \circ \bC^{\odot 2}_t
  \end{bmatrix},
\end{equation}
which corresponds to the previous definition \eqref{eq:W_def} where the sequence $(\bJ_n)$ and the process $\blambda_t$ are replaced by the constant matrix with all entries equal to one. We have the following.

\begin{theorem}
  \label{thm:concentration_continuous}
  Let $\bZ_t$ be given by~\eqref{eq:martingale} and suppose that
  Assumption~\ref{ass:M_c} holds. 
  Then, the following holds\textup:
  \begin{equation*}
    \P \bigg[ \normop{\bZ_t} \geq \sqrt{2v(x+\log(m+n))} ~,~\sigma^2( \bZ_t) \le v \bigg] \leq e^{-x}
  \end{equation*}
   for any $v, x > 0$, where $\sigma^2(\bZ_t)$ is defined in \eqref{eq:def_sigma_Z_t}. 
   Moreover, we have
  \begin{equation*}
  \sigma^2(\bZ_t) = \lmax(\bV_t),
  \end{equation*}
  where $\bV_t$ is given by
  \begin{equation}
  \label{eq:def_V_t_continuous}
  \bV_t = \int_0^t \bW_s \; ds,
  \end{equation}
  with $\bW_s$ given by Equation~\eqref{eq:W_def_cont}.  
\end{theorem}

We can remark that the above concentration inequality corresponds exactly to the result obtained in Theorem~\ref{thm:concentration_counting} for purely discontinuous martingales, if one sets $b_t = 0$. 
The concentration obtained here is in the ``Gaussian'' regime: $\bM_t$ is a Brownian motion, which leads to sub-Gaussian tails for $\bZ_t$. 
This is to be contrasted with Theorem~\ref{thm:concentration_counting}, which is in a ``Poisson'' regime: the tails contains both sub-Gaussian and sub-exponential terms for $\bZ_t$ in this case.

\subsection{Discussion}
\label{sec:discussion}

The concentration inequalities established for the two families of continuous-time matrix martingales considered above have the same form as the Freedman inequality obtained
in the discrete-time case~\cite{oliveira2010sums, tropp2011freedman}. 
In the case of deterministic functions $\tT_s$ and $\bC_s$ a direct consequence of Theorems~\ref{thm:concentration_counting} and~\ref{thm:concentration_continuous} is
\begin{equation}
\label{eq:mean_bound}
   \E \normop{\bZ_t} \leq \sigma(\bZ_t) \sqrt{2 \log(n+m)} + \frac{b_t \log(n+m)}{3},
\end{equation}
where $\sigma(\bZ_t)$ is defined by~\eqref{eq:def_sigma_Z_t_deter} and $b_t$ by~\eqref{eq:b_t_def}, with $b_t = 0$ if Assumption \ref{ass:M_c} holds.

By considering a piecewise constant tensor $\tT_s = \sum_{k=1}^n \tT_k \mathbbm{1}_{]k-1,k]}(s)$, where $\mathbbm{1}_{]k-1,k]}(t)$ stands for the indicator function of the interval $]k-1,k]$, $\bZ_t$ reduces to a discrete sum of random matrices $\bZ_n = \sum_{k=1}^n \bS_k$
with $\bS_k  = \tT_k \circ \int_{k-1}^{k} \bC_t \odot d\bM_t$. 
In this very particular case, one recovers exactly the results obtained by Tropp~\cite{tropp2012user} in this context, with a variance term given by
\begin{equation}
\label{eq:vartropp}
\sigma^2(\bZ_n) =  \max \bigg( \Big\| \sum_k \E(\bS_k \bS_k^\top) 
\Big\|_{\op}, \Big\| \sum_k\E(\bS_k^\top \bS_k) \Big\|_{\op} \bigg).
\end{equation}
Let us mention that, in the context of random series of matrices, a first tail bound for the norm was
provided by Ahlswede and Winter \cite{ahlswede2002strong}. 
These authors established a concentration inequality involving the variance term
\begin{equation*}
\sigma^2_{AW}(\bZ_n) = \max\Big( \sum_k \E \norm{ \bS_k \bS^\top_k}_{\op},  
\sum_k \E \norm{ \bS^\top_k \bS_k}_{\op} \Big), 
\end{equation*}
which is greater than the expression in Equation~\eqref{eq:vartropp}.
Ahlswede and Winter approach is based on the bounding of the matrix moment generating function $\xi \mapsto \E \tr e^{\xi \bZ_n}$  by iterating Golden-Thomson inequalities 
(which states that $\tr e^{\bA + \bB} \leq \tr e^{\bA} e^{\bB}$ for any symmetric matrices $\bA$ and~$\bB$). 
The improvement of $\sigma^2_{AW}(\bZ_n)$ to $\sigma^2(\bZ_n)$ obtained in~\cite{tropp2012user} is based on a powerful result by Lieb~\cite{lieb1973convex}, which says that $\bX \mapsto \tr e^{\bA + \log \bX}$ is concave over the SDP (semidefinite positive) cone, for any matrix $\bA$.

A surprising aspect of our results concerning continuous time martingales is that, as a by-product, they allow to recover previous sharp bounds without the use of the Lieb result.
The infinitesimal approach introduced in this paper allows one,
through It\^o's Lemma (see Appendix~\ref{sec:ito}), to bound 
$\E \tr \exp(\bZ_t)$ quite easily since it
can be explicitly written as $\E \int_0^t d (\tr \exp(\bZ_t))$.

As far as the sharpness of our results is concerned, better bounds than \eqref{eq:mean_bound}
can be manifestly obtained in some very specific cases. 
Indeed, it is well-known that for $n \times n$ matrices of symmetric  
i.i.d. Gaussian random variables (GOE ensemble), the expectation of the 
largest eigenvalue is of order $\sqrt{n}$. 
This result has been extended to more general
matrices of i.i.d. random variables as, e.g., in the work of Seginer\cite{seginer2000} or Latala~\cite{latala2005} where bounds without the $\sqrt{\log n}$ factor are obtained. 

However, for the general case considered in this paper, our results can be considered as being sharp since they match inequalities from~\cite{tropp2012user} on several important particular cases.
We develop some of these cases in Section~\ref{sec:Applications} below.
Concerning the extra $\log (m+n)$ factor, a simple example is the case of a diagonal matrix of i.i.d. standard Gaussian variables.
In that case the largest eigenvalue is simply the maximum of $m + n$, whose expectation
is well-known to scale as $\sqrt{\log(m + n)}$. 
We refer the reader to the discussion in~\cite{bandeiraMit,tropp2012user} for further details.

\section{Some specific examples}
\label{sec:Applications}

In this section we provide some examples of applications of Theorems~\ref{thm:concentration_counting} and 
\ref{thm:concentration_continuous} and discuss their relationship with some former works.
Further generalizations in an even more general context and application to statistical 
problems are then briefly presented.

\subsection{The martingale $\bZ_t = \int_0^t \bA_s (\bC_s \odot d\bM_s) \bB_s$}

Let $\bA_s$ and $\bB_s$ two matrix-valued processes of bounded predictable functions of dimensions respectively $m \times p$ and $q \times n$. 
Let us suppose that $(\tT_s)_{i,j,k,l} = (\bA_s)_{i,k} (\bB_s)_{l,j}$.
This corresponds to the situation where the matrix martingale $\bZ_t$ can be written as
\begin{equation}
\label{eq:mart_AB}
\bZ_t = \int_0^t \bA_s (\bC_s\odot d \bM_t) \bB_t.
\end{equation}
In that case, the entry $(i,j)$ of the matrix $\bW_s$ defined by~\eqref{eq:W_def} reads, 
when $1 \leq i,j \leq m$:
\begin{align*}
  (\bW_s)_{i,j} = & \sum_{a=1}^n \sum_{k=1}^p \sum_{l=1}^q (\tT_s)_{i,a,k,l}(\tT_s)_{j,a,k,l} 
  (\E(\bJ_1^{\odot 2}) \odot \bC^{\odot 2}_s \odot \blambda_s))_{k,l} \\
  = & \sum_{a=1}^n \sum_{k=1}^p \sum_{l=1}^q (\bA_s)_{i,k} (\bB_s)_{l,a}^2 (\bA_s)_{j,k} 
  (\E(\bJ_1^{\odot 2}) \odot \bC^{\odot 2}_s \odot \blambda_s))_{k,l}.
\end{align*}
In the same way, for $1 \leq i,j \leq n$, one has:
$$
  (\bW_s)_{i+m,j+m} = \sum_{a=1}^m \sum_{k=1}^p \sum_{l=1}^q (\bB_s)_{k,i} (\bA_s)_{a,l}^2 (\bB_s)_{k,j} (\E(\bJ_1^{\odot 2}) \odot \bC^{\odot 2}_s \odot \blambda_s))_{l,k} \; .
$$
Then, Theorem~\ref{thm:concentration_counting} leads to the following corollary, that follows from easy computations.
\begin{proposition}
\label{prop:AMB}
If $\bZ_t$ is given by~\eqref{eq:mart_AB}, the matrix $\bW_s$ defined by~\eqref{eq:W_def} 
can be written  as
\begin{equation}
\label{eq:W_def_old}
\bW_t = 
\bP_t^\top
\begin{bmatrix}
\diag[(\E(\bJ_1^{\odot 2}) \odot \bC^{\odot 2}_t \odot \blambda_t) \diag[\bB_t \bB_t^\top] \dsone] & \bO \\ \bO &
\diag[(\E(\bJ_1^{\odot 2}) \odot \bC^{\odot 2}_t \odot \blambda_t)^\top \diag[\bA_t^\top \bA_t] \dsone \big] 
\end{bmatrix}
\bP_t
\end{equation}
with
\begin{equation*}
\bP_t =
\begin{bmatrix}
\bA_t^\top & \bs 0 \\ \bs 0 &
\bB_t
\end{bmatrix}.
\end{equation*}
Furthermore, we have also that~\eqref{eq:b_t_def} writes
\begin{equation}
\label{eq:b_t_old}
b_t = J_{\max} \sup_{s \in [0, t]}  \norm{\bA_s}_{\infty, 2} 
\norm{\bB_s}_{2,\infty} \norm{\bC_s}_\infty.
\end{equation}
The same expression holds for the matrix $\bW_t$ of Theorem~\ref{thm:concentration_continuous}, 
provided that one takes $(\blambda_t)_{i,j} = (\bJ_1)_{i,j} = 1$.
\end{proposition}

The particular structure~\eqref{eq:mart_AB} enables the study of several particular cases, developed in the next sections.

\subsubsection{Counting processes}

An interesting example that fits the setting of purely discontinuous martingales is the situation where $\bM_t$ comes from the compensation of a matrix-valued process $\bN_t$, whose entries are counting processes.
In this example we can write $\bM_t = \bN_t = \bLambda_t$, where $\bLambda_t$ is the compensator 
of $\bN_t$.
We fix $\bA_t = \bI_p$ and $\bB_t = \bI_q$ for all $t$, so that $\bZ_t =
\int_0^t \bC_s \odot d \bM_s$. We obtain in this case 
\begin{equation*}
\bV_t = \int_0^t
\begin{bmatrix}
\diag\big[ (\bC_s^{\odot 2} \odot \blambda_s) \dsone \big] & \bO
\\ \bO & \diag\big[ (\bC_s^{\odot 2} \odot \blambda_s)^\top \dsone
\big]
\end{bmatrix} ds,
\end{equation*}
so the largest eigenvalue is easily computed as
\begin{equation*}
\lmax(\bV_t) = \Big\| \int_0^t \bC_s^{\odot 2} \odot \blambda_s ds
\Big\|_{1, \infty} \vee \Big\| \int_0^t \bC_s^{\odot 2} \odot \blambda_s ds
\Big\|_{\infty, 1},
\end{equation*}
and $b_t = \sup_{s \in [0, t]} \norm{\bC_s}_\infty$. This leads to the
following corollary.
\begin{corollary}
\label{cor:1}
	Let $\{ \bN_t \}$ be a $p \times q$ matrix whose entries
	$(\bN_t)_{i, j}$ are independent counting processes
	with intensities $(\blambda_t)_{i, j}$. Consider the matrix martingale
	$\bM_t = \bN_t - \bLambda_t$, where $\bLambda_t = \int_0^t
	\blambda_s ds$ and let $\{ \bC_t \}$ be a $p \times q$ bounded deterministic process. We have that
	\begin{align*}
	\Bignormop{\int_0^t \bC_s \odot d(\bN_t - \bLambda_t)} \leq
	&\sqrt{2 \Big( \Big\| \int_0^t \bC_s^{\odot 2} \odot \blambda_s ds
		\Big\|_{1, \infty} \vee \Big\| \int_0^t \bC_s^{\odot 2} \odot
		\blambda_s ds \Big\|_{\infty, 1} \Big) (x + \log(p + q))} \\
	& \quad \quad + \frac{\sup_{s \in [0, t]} \norm{\bC_s}_\infty (x +
		\log(p + q))}{3}
	\end{align*}
	holds with a probability larger than $1 - e^{-x}$.
\end{corollary}

Another interesting situation is when $\bA_s$ is of dimension $1 \times q$, $\bM_t$ is of dimension $q \times 1$, $\bC_s$ is the matrix of dimension $q \times 1$ will all entries equal to one, $\bB_t = 1$ for all $t$. 
In that case $\bZ_t$ is a scalar martingale denoted $Z_t$.
Consider $A_t^{(1)}, \ldots, A^{(q)}_t$  and $N^{(1)}_t, \ldots, N^{(q)}_t$ the $q$ components of respectively the vector $\bA_t^{\top}$ and the vector $\bN_t$. 
Along the same line $\lambda^{(1)}_t, \ldots, \lambda^{(q)}_t$ denotes their associated intensities.
We thus have
\begin{equation*}
  Z_t = \sum_{k=1}^q \int_0^t A^{(k)}_s dN^{(k)}_s,
\end{equation*}
which is a martingale considered in~\cite{hansen_reynaud_bouret_viroirard}.
In this case we have from Proposition~\ref{prop:AMB}:
\begin{equation*}
\bV_t = \int_0^t
\begin{bmatrix}
 \sum_{k=1}^q (A_s^{(k)})^2 \lambda^{(k)}_s & 0
\\ 0 &  \sum_{k=1}^q (A_s^{(k)})^2 \lambda^{(k)}_s
\end{bmatrix} ds,
\end{equation*}
which largest eigenvalue is simply $\int_0^t \sum_{k=1}^q (A_s^{(k)})^2 \lambda^{(k)}_s ds$.
Theorem~\ref{thm:concentration_counting} becomes in this particular case the following.
\begin{corollary}
	Let $(N^{(1)}_t, \ldots, N^{(q)}_t)$ be $q$ counting processes of intensities 
	$\lambda^{(1)}_t, \ldots, \lambda^{(q)}_t$. Let us consider the martingale
  \begin{equation*}
    Z_t = \sum_{k=1}^q \int_0^t A^{(k)}_s (dN^{(k)}_s - \lambda_s^{(k)} ds)
  \end{equation*}
	where $(A^{(k)})_{k=1,\ldots, q}$ are $q$ predictable functions.
    If one assumes that $b_t = \sup_{k,s\leq t} \norm{A^{(k)}_s} \leq 1$, 
    the following inequality
    \begin{equation*}
    \P \bigg( | Z_t| \geq \sqrt{2 x v} + \frac{x}{3}, \quad 
    \sum_{k=1}^q \int_0^t (A^{(k)}_s)^2 \lambda^{(k)}_s ds \leq v \bigg) 
    \leq 2 e^{-x}
    \end{equation*}
    holds for any $x,v >0$.
\end{corollary}
This result exactly corresponds to the concentration inequality proved in \cite{hansen_reynaud_bouret_viroirard} in the context of statistical estimation
of point processes.

\subsubsection{``Static'' random matrices}

Theorems \ref{thm:concentration_counting} and \ref{thm:concentration_continuous} can be helpful
to study the norm of some specific random matrices.

Let us consider a $n \times m$ matrix $\bG = [g_{i, j}]$ of independent centered 
Gaussian random variables $g_{i, j}$with variance $c_{i,j}^2$.
This corresponds to the situation in Proposition~\ref{prop:AMB} when $t=1$, $
\bA_t = \bI_n$, $\bB_t = \bI_m$ and $(\bC_t)_{i,j} = (\bC)_{i,j} = c_{i,j}$.
The $(n + m) \times (n + m)$ matrix $\bW_t$ given by~\eqref{eq:W_def} writes in this case as the diagonal matrix with entries equal to the square $\ell^2$-norms of respectively rows and 
columns of $\bC$. In this setting, Theorem~\ref{thm:concentration_continuous} entails the following.
\begin{corollary}
	\label{cor:grm}
	Let $\bG$ be a $n \times m$ random matrix with independent entries $g_{i, j}$ that are centered Gaussian with variance $c_{i,j}^2$.
  Then,
	\begin{equation}
	\label{eq:gbt}
	\P \Big( \normop{\bG} \geq \sigma \sqrt{x + \log(n + m)} \Big) \leq e^{-x}
	\end{equation}
	with
	\begin{equation*}
  	\sigma^2 = \max\big(\norm{\bC}_{\infty, 2}, \norm{\bC}_{2, \infty}\big) =
    \max\bigg(\max_{i=1, \ldots, n} \sum_{j=1}^m c^2_{i, j}, \max_{j=1, \ldots, m} 
    \sum_{i=1}^n c^2_{i, j} \bigg).
	\end{equation*}
\end{corollary}
In the case of standard Gaussian random variables, i.e., $b_{i,j}^2 = 1$, 
we simply have $\sigma^2= \max(n, m)$.
Moreover, Equation~\eqref{eq:gbt} entails in the case $n=m$:
\begin{equation*}
  \E \normop{\bG} \leq \sigma \sqrt{2 \log(2 n)}.
\end{equation*}
We therefore recover the bounds on $\E{\normop{\bG}}$ that results from 
concentration inequalities obtained by alternative methods \cite{oliveira2010sums, tropp2012user}.
We refer the reader to Section~\ref{sec:discussion} and~\cite{tropp2012user} for a discussion about the sharpness of this result.

The same kind of result can be obtained for a random matrix $\bN$ containing independent entries with a Poisson distribution. 
Take $\bC_t =\bC$ as the $n \times m$ matrix with all entries equal to one, and consider
the $n \times m$ matrix $\bN_t$ with entries $(\bN_t)_{i, j}$ that are homogeneous Poisson processes on $[0, 1]$ with (constant) intensity $\lambda_{i, j}$.
Taking $t = 1$, and forming the matrix $\blambda$ with entries $(\blambda)_{i, j} = \lambda_{i, j}$, we obtain from Corollary~\ref{cor:1} the following.
\begin{corollary}
	Let $\bN$ be a $n \times m$ random matrix whose entries $(\bN)_{i, j}$ have a Poisson distribution with intensity $\lambda_{i, j}$. Then, we have
	\begin{equation*}
	\P \bigg( \normop{\bN - \blambda} \geq \sqrt{2 (\| \blambda \|_{1, \infty}
		\vee \| \blambda \|_{\infty, 1})x } + \frac{x}{3} \bigg) \leq (n+m) e^{-x}
	\end{equation*}
  for any $x > 0$, where $\blambda$ has entries $(\blambda)_{i, j} = \lambda_{i, j}$.
\end{corollary}

Such a result for random matrices with independent Poisson entries was
not, up to the knowledge of the authors, explicitly exhibited in
literature. Note that, in contrast to the Gaussian case considered in Corollary~\ref{cor:grm}, the variance term depends on the maximum $\ell_1$ norm of rows and columns of $\blambda$, which comes from the subexponentiality of the Poisson distribution.

\subsection{Stochastic integral of a matrix of functions}

In this section we consider the simple case where $\bM_s= M_s $ is scalar martingale 
and $\tT_s$ is a matrix of deterministic functions, i.e., $(\tT_s)_{i,j;k,l} = (\bA_s)_{i,j}$.
Let us suppose, for the sake of simplicity, that $\bC_s=1$.
The matrix martingale $\bZ_t$ therefore writes
\begin{equation}
\label{newmart}
 \bZ_t = \int_0^t \bA_s dM_s.
\end{equation}
If that case, Theorems~\ref{thm:concentration_counting} and~\ref{thm:concentration_continuous} lead to the following.

%where $\tA_s$ stands for a rank 3 tensor of dimension $(n \times m \times q)$ 
%and $\bC_s$ and $\bM_s$ are vectors of dimension $q$.
%The tensor $\tA$ can be conveniently represented by the transpose of a vector of size $q$
%of matrices:
%$$
%\tA = \left(\bA_1,\ldots,\bA_q \right)
%$$
%with in particular
%$$
%\tA_{i,j,k} = (\bA_k)_{i,j}
%$$
%According to this convention, the matrix corresponding to the product of $\tA$ and a vector $\bB$ of dimension $q$ reads:
%$$
%\tA \bB = \sum_{k=1}^q B_k \bA_k
%$$
%It is easy to show that, within this framework, all the proofs of Theorem \ref{thm:concentration_counting} and \ref{thm:concentration_continuous}
%remain valid provided one replaces the coefficients $A_{1,k}$ and the matrix $\bA_k$.
%We then have the following proposition:
%
%\begin{proposition}
%Theorems \ref{thm:concentration_counting} and \ref{thm:concentration_continuous} remain valid for the martingale \eqref{newmart} in respectively
%purely discontinuous and continuous cases, provided
%one replaces $\bA_t$ by $\tA_t$ and $\bB_t$ by $1$ in expressions \eqref{eq:W_def} and \eqref{eq:W_def_cont} .
%This leads to the following expression for the matrix $\bW_t$:
%\begin{equation}
%\label{eq:W_new}
%\bW_t = \begin{bmatrix}
%\sum_{k=1}^q  [(\lambda_k)_t] (C_k)_t^2 (\bA_k)_t (\bA_k)_t^\top & \bO \\ \bO &
%\sum_{k=1}^q [(\lambda_k)_t ](C_k)_t^2 (\bA_k)_t^\top (\bA_k)_t
%\end{bmatrix},
%\end{equation}	
%where the factors $(\lambda_k)_t$ are only present in the discontinuous case.
%\end{proposition} 
%
%
%
%A direct consequence of this result, in the case $q=1$ and $M_t = W_t$ the standard Brownian motion:

\begin{proposition}
\label{prop:troppcont}
Let $b_t = \sup_{s \in [0, t]} \max(\norm{\bA_s}_{2, \infty}, 
\norm{\bA_s}_{\infty, 2})$ if $M_t$ satisfies Assumption \ref{ass:M_d} and take $b_t = 0$ if $M_t$ is a Brownian motion.
Let us define the variance
\begin{equation}
\label{sigma2TroppCont}
\sigma_t^2 = \max \bigg( \Big\| \int_0^t \bA_s \bA^\top_s ds \Big\|_{\op}, 
          \Big\| \int_0^t \bA^\top_s \bA_s ds \Big\|_{\op} \bigg). 
\end{equation}
Then 
\begin{equation}
\label{eq:rgs}
\P \Big( \normop{\bZ_t} \geq \sqrt{2 \sigma_t^2 x} + \frac{x b_t}{3} \Big) \leq (n+m) e^{-x}
\end{equation}
for any $x > 0$.
\end{proposition}

This result is a continuous time version of an analogous inequality obtained
in~\cite{tropp2012user} for series of random matrices $\bZ_n$ of the form
\begin{equation*}
 \bZ_n = \sum_{k=1}^n \gamma_k \bA_k,  
\end{equation*}
where $\gamma_k$ are i.i.d. zero mean random variables (e.g. standard normal) and $\bA_k$ is a sequence of deterministic matrices. Note that Proposition~\ref{prop:troppcont} allows to recover the result for a discrete sequence $\bZ_n$ simply by considering a piecewise constant matrix-valued process $\bA_s$.

\section*{Acknowledgments}
We gratefully acknowledge the anonymous reviewers of the first version of this paper for their 
helpful comments and suggestions. 
The authors would like to thank Carl Graham and Peter Tankov for various comments on our paper.
This research benefited from the support of the Chair ``Markets in Transition'',
under the aegis of ``Louis Bachelier Finance and Sustainable Growth'' laboratory, a joint initiative of
\'Ecole Polytechnique,
Universit\'e d'\'Evry Val d'Essonne and
F\'ed\'eration Bancaire
Francaise.

\appendix

\counterwithin{lemma}{section}
\counterwithin{proposition}{section}

\section{Tools for the study of matrix martingales in continuous time}
\label{sec:tools}

In this section we give tools for the study of matrix martingales in
continuous time. We proceed by steps. The main result of this section, namely
Proposition~\ref{prop:new-counting-supermart}, proves that the
trace exponential of a matrix martingale is a supermartingale, when
properly corrected by terms involving quadratic covariations.

\subsection{A first tool}

We give first a simple lemma
that links the largest eigenvalues of random matrices to the trace
exponential of their difference.

\begin{lemma}
  \label{lem:deviation}
  Let $\bs X$ and $\bs Y$ be two symmetric random matrices such that
  \begin{equation*}
    \tr \E[ e^{\bX - \bY}] \leq k
  \end{equation*}
  for some $k > 0$.  Then, we have
  \begin{equation*}
    \P[ \lmax(\bX) \geq \lmax(\bY) + x ] \leq k e^{-x}
  \end{equation*}
  for any $x > 0$.
\end{lemma}
\begin{proof}
Using the fact that \cite{petz1994survey}
\begin{equation}
  \label{eq:monotone_tr_exp}
  \bA \mleq \bB \Rightarrow \tr \exp(\bA) \leq \tr \exp(\bB),~~\mbox{for any}~\bA,\bB~\mbox{symmetric},
\end{equation}
along with the fact that $\bY \mleq \lmax(\bY) \bI$, one has
  \begin{equation*}
    \tr \exp(\bs X - \bs Y) \ind{E} \geq \tr
    \exp(\bs X - \lmax(\bY) \bs I) \ind{E},
  \end{equation*}
  where we set
$E = \{ \lmax(\bX) \geq \lmax(\bY) + x \}$.
  Now, since $\lmax(\bM) \leq \tr \bM$ for any symmetric positive definite matrix $\bM$, we
  obtain
  \begin{align*}
    \tr \exp(\bs X - \bs Y) \ind{E} &\geq \lmax(\exp(\bs X -
    \lmax(\bY) \bs I)) \ind{E} \\
    &= \exp(\lmax(\bs X) - \lmax(\bs Y)) \ind{E} \\
    &\geq e^x \ind{E},
  \end{align*}
  so that taking the expectation on both sides proves
  Lemma~\ref{lem:deviation}.
\end{proof}

\subsection{Various definitions and It\^o's Lemma for functions of matrices}
\label{sec:ito}

In this section we describe some classical notions from stochastic calculus~\cite{jacodshi,liptser1989theory}  and extend them to matrix semimartingales. 
Let us recall that the {\em quadratic covariation} of two scalar semimartingales $X_t$ and $Y_t$ is defined as 
\begin{equation*}
[X,Y]_t = X_t Y_t - \int_0^t Y_{t^-} dX_t - \int_0^t X_{t^-} dY_t - X_0 Y_0 \; .
\end{equation*}
It can be proven (see e.g. \cite{jacodshi}) that the non-decreasing process  $[X, X]_t$, often denoted as $[X]_t$, does correspond to the 
{\em quadratic variation} of $X_t$ since it is equal to the limit (in probability) of $\sum_i (X_{t_i}-X_{t_{i-1}})^2$ when the mesh size 
of the partition $\{t_i \}_i$ of the interval $[0,t]$ goes to zero. 

If $X_t$ is a square integrable scalar martingale, then its {\em predictable quadratic variation} $\inr{X}_t$ is defined as the unique predictable increasing process such that $X_t^2 - \inr{X}_t$ 
is a martingale. 
The predictable quadratic covariation between two square integrable scalar martingales $X_t$ and $Y_t$ is then defined from the polarization identity:
\begin{equation*}
  \inr{X ,Y} = \frac{1}{4} \big( \inr{X + Y, X + Y} - 
  \inr{X - Y, X - Y} \big).
\end{equation*} 
A martingale $X_t$ is said to be \emph{continuous} if its sample paths $t\mapsto X_t$ are a.s. continuous, and {\em purely discontinuous}\footnote{Let us note that this definition does not imply that a purely discontinuous martingale is the sum of its jumps: for example a compensated Poisson process $N_t-\lambda t$ is a purely discontinuous martingale that has a continuous component.} 
if $X_0 = 0$ and $\inr{X, Y}_t = 0$ for any continuous martingale $
Y_t$.

The notion of predictable quadratic variation can be extended to semimartingales. Indeed, any semimartingale 
$X_t$ can be represented as a sum:
\begin{equation}
  \label{eq:dec}
  X_t = X_0 + X^{c}_t + X^{d}_t + A_t,
\end{equation}
where $X^{c}_t$ is a continuous local martingale, $X^{d}_t$ is a purely discontinuous local martingale 
and $A_t$ is a process of bounded variations. 
Since in the decomposition~\eqref{eq:dec}, $X^{c}_t$ is 
unambiguously determined, $\inr{X^c}_t$ is therefore well 
defined~\cite{jacodshi}.
Within this framework, one can prove (see e.g.~\cite{jacodshi}) that if 
$X_t$ and $Y_t$ are two semimartingales, then:
\begin{equation}
  \label{qv}
  [X,Y]_t = \inr{X^c,Y^c}_t + \sum_{0\leq s \leq t} 
  \Delta X_s \Delta Y_s.
\end{equation}

All these definitions can be naturally extended to matrix valued semimartingales.
Let $\bX_t$ be a $p \times q$ matrix whose entries are real-valued square-integrable semimartingales. 
We denote by $\inr{\bM}_t$ the matrix of entry-wise predictable quadratic variations. 
The predictable quadratic covariation of $\bX_t$ is defined with the help of the vectorization operator $\vec : \R^{p \times q} \rightarrow \R^{pq}$ which stacks vertically the columns of $\bX$, namely if $\bX \in
\R^{p \times q}$ then
\begin{equation*}
  \vec(\bX) =
  \begin{bmatrix}
    \bX_{1, 1} \cdots \bX_{p, 1} \bX_{1, 2} \cdots \bX_{p, 2} \cdots
    \bX_{1, q} \cdots \bX_{p, q}
  \end{bmatrix}^\top.
\end{equation*}
We define indeed the predictable quadratic covariation matrix $\inr{\vec \bX}_t$ of $\bX_t$ as the $pq \times pq$ matrix with entries
\begin{equation}
  \label{eq:predictable_covariation}
  (\inr{\vec \bX}_t)_{i,j} = \inr{(\vec \bX_t)_i, (\vec \bX_t)_j}
\end{equation}
for $1 \leq i, j \leq pq$, namely such that $\vec(\bX_t) \vec(\bX_t)^\top - \inr{\vec \bX}_t$ is a martingale.
The matrices of quadratic variations $[\bX]_t$ and quadratic covariations $[\vec \bX]_t$ 
are defined along the same line. 

Then according to Equation~\eqref{qv}, we have:
\begin{equation}
  \label{eq:quadratic-variation}
  [\bX]_t = \inr{\bX^c}_t + \sum_{0 \leq s \leq t} (\Delta \bX_s)^2,
\end{equation}
and
\begin{equation*}
  [\vec \bX]_t = \inr{\vec \bX^c}_t + \sum_{0 \leq s \leq t} \vec(\Delta \bX_{s}) \vec(\Delta \bX_{s})^\top.
\end{equation*}
An important tool for our proofs is It\^o's lemma, that allows to
compute the stochastic differential $d F(\bM_t)$ 
where $F : \R^{p \times q} \rightarrow \R$ is a twice 
differentiable function. 
We denote by $\frac{d F}{d \vec(\bX)}$
the $p q$-dimensional vector such that
\begin{equation*}
  \Big[ \frac{d F}{d \vec(\bX)} \Big]_i = \frac{\partial F}{\partial
    (\vec \bX)_i} \; \text{ for } \; 1 \leq i \leq p q.
\end{equation*}
The second order derivative is the $pq \times pq$ symmetric matrix
given by
\begin{equation*}
  \Big[ \frac{d^2 F}{d \vec(\bX)d \vec(\bX)^\top}\Big]_{i, j} =
  \frac{\partial^2 F}{\partial (\vec \bX)_i \partial (\vec \bX)_j}
  \; \text{for} \; 1 \leq i, j \leq pq.
\end{equation*}
A direct application of the multivariate It\^o Lemma (\cite{liptser1989theory} Theorem~1, p.~118) writes for
matrix semimartingales as follows.
\begin{lemma}[It\^o's Lemma]
  \label{lem:ito}
  Let $\{\bX_t\}_{t \geq 0}$ be a $p \times q$ matrix semimartingale and $F : \R^{p \times q} \rightarrow
  \R$ be a twice continuously differentiable function. Then
  \begin{align*}
    dF(\bX_t) &= \Big(\frac{d F}{d \vec(\bX)}(\bX_{t^-}) \Big)^\top
    \vec(d \bX_t) + \Delta F(\bX_t) - \Big(\frac{d F}{d \vec
      \bX}(\bX_{t^-}) \Big)^\top \vec (\Delta \bX_t) \\
    & \quad + \frac 12 \tr \bigg(\Big(\frac{d^2 F}{d \vec(\bX) d
      \vec(\bX)^\top} (\bX_{t^-})  \Big)^\top d\inr{\vec \bX^c}_t 
      \bigg).
  \end{align*}
\end{lemma}
As an application, let us apply Lemma~\ref{lem:ito} to the function $F(\bX)= \tr \exp(\bX)$ that acts on the set of symmetric matrices. This result will be of importance for the proof of our results.
\begin{lemma}[It\^o's Lemma for the trace exponential]
  \label{lem:ito-trace-exp}
  Let $\{\bX_t\}$ be a $d \times d$ symmetric matrix semimartingale.  The It\^o formula for $F(\bX_t) = \tr
  \exp(\bX_t)$ gives
  \begin{equation}
    % \label{eq:ito-trace-exp}
    d(\tr e^{\bX_t})  = \tr(e^{\bX_{t^-}} d \bX_t) + \Delta(\tr
    e^{\bX_t}) - \tr(e^{\bX_{t^-}} \Delta \bX_t) + \frac 12 \sum_{i=1}^d
    \tr(e^{\bX_{t^-}} d\inr{\bX_{\bul, i}^c}_t),
  \end{equation}
  where $\inr{\bX_{\bul, i}^c}_t$ denotes the $d \times d$ predictable quadratic variation of the continuous part of the $i$-th column $(\bX_t)_{\bul, i}$ of $\bX_t$.
\end{lemma}

\begin{proof}
  An easy computation gives
  \begin{equation*}
    \tr e^{\bX + \bH} = \tr e^{\bX} + \tr(e^{\bX} \bH) + \tr(e^{\bX}
    \bH^2) + \text{ higher order terms in } \bH
  \end{equation*}
  for any symmetric matrices $\bX$ and $\bH$. Note that $\tr(e^\bX \bH) =
  (\vec \bH)^\top \vec(e^\bX) $, and we have
  from~\cite{johnson1991topics} Exercise~25 p.~252 that
  \begin{equation*}
    \tr(e^{\bX} \bH^2) = \tr(\bH e^{\bX} \bH) = (\vec
    \bH)^\top (\bI \otimes e^\bX) (\vec \bH),
  \end{equation*}
  where the Kronecker product $\bI \otimes e^\bX$ stands for the block
  matrix
  \begin{equation*}
    \bI \otimes \bY =
    \begin{bmatrix}
      e^\bX & 0 & \cdots & 0 \\
      0 & e^\bX &  & \vdots \\
      \vdots &  & \ddots & 0 \\
      0 & \cdots & 0 & e^\bX
    \end{bmatrix}.
  \end{equation*}
  This entails that
  \begin{equation*}
    \frac{d(\tr e^{\bX})}{d \vec(\bX)} = \vec(e^{\bX}) \;\; \text{ and }
    \;\; \frac{d^2(\tr e^{\bX})}{d \vec(\bX) d \vec(\bX)^\top} = \bI
    \otimes e^\bX.
  \end{equation*}
  Hence, using Lemma~\ref{lem:ito} with $F(\bX) = \tr e^{\bX}$ we obtain
  \begin{align*}
    d(\tr e^{\bX_t}) &= \vec(e^{\bX_{t^-}})^\top \vec(d \bX_t) +
    \Delta(\tr e^{\bX_t}) - \vec(e^{\bX_{t^-}})^\top \vec(\Delta
    \bX_t) \\
    & \quad + \frac 12 \tr \big( (\bI \otimes e^{\bX_{t^-}})
   d\inr{\vec \bX^c}_t \big).
    \end{align*}
    Since $\vec(\bY)^\top \vec(\bZ) = \tr(\bY\bZ)$, one gets
    \begin{align*}
      d(\tr e^{\bX_t}) &= \tr(e^{\bX_{t^-}} d \bX_t) + \Delta(\tr
      e^{\bX_t}) - \tr(e^{\bX_{t^-}} \Delta \bX_t) + \frac 12 \tr\big(
      (\bI \otimes e^{\bX_{t^-}}) d\inr{\vec \bX^c}_t \big).
  \end{align*}
  To conclude the proof of Lemma~\ref{lem:ito-trace-exp},
  it remains to prove that
  \begin{equation*}
    \tr\big( (\bI \otimes e^{\bX_{t^-}}) d \inr{\vec \bX^c}_t \big) = \sum_{i=1}^d \tr(e^{\bX_{t^-}} d\inr{\bX_{\bul, i}^c}_t).
  \end{equation*}
  First, let us 
  write
  \begin{equation*}
    d\inr{\vec \bX^c}_t  = \sum_{1 \leq i, j \leq d} \bE^{i, j} \otimes
    d \inr{\bX_{\bul, i}^c, \bX_{\bul, j}^c}_t,
  \end{equation*}
  where $\bE^{i, j}$ is the $d \times d$ matrix with all entries equal
  to zero excepted for the $(i, j)$-entry, which is equal to one.
  Since
  \begin{equation*}
    (\bA \otimes \bB)(\bC \otimes \bD) = (\bA \bC) \otimes (\bB \bD)
    \quad \text{and} \quad \tr(\bA \otimes \bB) = \tr(\bA) \tr(\bB)
  \end{equation*}
  for any matrices $\bA, \bB, \bC, \bD$ with matching dimensions (see
  for instance~\cite{johnson1991topics}), we have
  \begin{equation*}
    \tr \big((\bI \otimes e^{\bX_{t^-}}) d\inr{\vec \bX^c}_t \big) =  \sum_{1
      \leq i, j \leq d} \tr(\bE^{i, j}) \tr(e^{\bX_{t-}} d \inr{\bX_{\bul,
      i}^c,\bX_{\bul, j}^c}_t) =  \sum_{i=1}^d \tr(e^{\bX_{t-}} d\inr{\bX_{\bul, i}^c, \bX_{\bul, i}^c}_t)
  \end{equation*}
  since $\tr \bE^{i, j} = 0$ for $i \neq j$ and $1$ otherwise. This
  concludes the proof of Lemma~\ref{lem:ito-trace-exp}.
\end{proof}

\subsection{A matrix supermartingale property}

The next proposition is a key property that is used below for the proofs of concentration inequalities both for purely discontinuous  and continuous matrix martingales.

\begin{proposition}
  \label{prop:new-counting-supermart}
Let $\{ \bY_t \}_{t \geq 0}$ be a $d \times d$ symmetric matrix martingale such that $\bY_0 = \bO$ and whose
entries are locally bounded.
Let $\bU_t$ be defined by
 \begin{equation}
    \label{eq:ass_counting-supermart}
    \bU_t =  \sum_{s \leq t} \big(e^{\Delta \bY_{s}} - \Delta \bY_{s}
    - \bI \big).
  \end{equation}
  If the matrix $\bU_t$ has an entry-wise compensator $\bA_t$ \textup(i.e., $\bU_t - \bA_t$ is a matrix martingale\textup) which is predictable, continuous and has finite variation \textup(FV\textup) then the process
  \begin{equation}
    L_t = \tr \exp \Big(\bY_t - \bA_t - \frac 12 \sum_{j=1}^d
    \inr{\bY_{\bul, j}^c}_t\Big)
  \end{equation}
  is a supermartingale. In particular, we have $\E L_t \leq d$ for any $t \geq 0$.
\end{proposition}

Proposition~\ref{prop:new-counting-supermart} can be understood as an extension to random matrices of the exponential supermartingale property given implicitly in the proof of Lemma~2.2 in~\cite{van1995exponential}, or the supermartingale property for multivariate counting processes from~\cite{bremaud1981point}, see Theorem~2, p.~165 and in Chapter~4.13 from~\cite{liptser1989theory}.

\begin{proof}
  Define for short
  \begin{equation*}
    \bX_t = \bY_t - \bA_t - \frac 12 \sum_{j=1}^d \inr{\bY_{\bul,j}^c}_t.
  \end{equation*}
Since $\bA_t$ and $\inr{\bY_{\bul, j}^c}_t$ for $j=1, \ldots,
  d$ are FV processes, then
  \begin{equation}
    \label{eq:quad_var_X_Y_equal}
    \inr{\vec \bX^c} = \inr{\vec \bY^c}
   \end{equation}
  and in particular $\inr{\bY_{\bul, j}^c} = \inr{\bX_{\bul, j}^c}$ for any $j=1, \ldots, d$. Using Lemma~\ref{lem:ito-trace-exp}, one has that for all $t_1 < t_2$:
  \begin{align*}
    L_{t_2} - L_{t_1} &= \int_{t_1}^{t_2} \tr(e^{\bX_{t^-}} d \bX_t) +
    \sum_{t_1 \leq t \leq t_2} \big(\Delta(\tr e^{\bX_{t}}) -
    \tr(e^{\bX_{{t}^-}} \Delta \bX_{t}) \big) \\
    & \quad \quad + \frac 12 \int_{t_1}^{t_2} \sum_{j=1}^d
    \tr(e^{\bX_{t^-}} d\inr{\bX_{\bul,j}^c}_t) \\
    &= \int_{t_1}^{t_2} \tr(e^{\bX_{t^-}} d \bY_t) - \int_{t_1}^{t_2}
    \tr(e^{\bX_{t^-}} d \bA_t) \\
    & \quad \quad + \sum_{t_1 \leq t \leq t_2} \big( \tr(
    e^{\bX_{{t}^-} + \Delta \bY_{t}}) - \tr(e^{\bX_{{t}^-}})
    -\tr(e^{\bX_{t^-}} \Delta \bY_{t}) \big),
  \end{align*}
  where we used~\eqref{eq:quad_var_X_Y_equal} together with the fact
  that $\Delta \bX_t = \Delta \bY_t$, since $\bA_t$ and
  $\inr{\bY_{\bul, j}^c}_t$ are both continuous.

  The Golden-Thompson's inequality, see~\cite{bhatia1997matrix},
  states that $\tr e^{\bA + \bB} \leq \tr( e^\bA e^\bB)$ for any
  symmetric matrices $\bA$ and $\bB$. Using this inequality we get
  \begin{align*}
    L_{t_2} - L_{t_1} &\leq \int_{t_1}^{t_2} \tr(e^{\bX_{t^-}} d
    \bY_t) - \int_{t_1}^{t_2} \tr(e^{\bX_{t^-}} d \bA_t) + \sum_{t_1
      \leq t \leq t_2} \tr \big( e^{\bX_{{t}^-}}(e^{\Delta
      \bY_{t}} -\Delta \bY_{t} - \bI) \big)  \\
    &= \int_{t_1}^{t_2} \tr(e^{\bX_{t^-}} d \bY_t) + \int_{t_1}^{t_2}
    \tr \big(e^{\bX_{t^-}}d(\bU_t - \bA_t) \big).
  \end{align*}
  Since $\bY_t$ and $\bU_t - \bA_t$ are matrix martingales, 
  $e^{\bX_{t^-}}$ is a predictable process with locally bounded entries
  and $L_t \geq 0$, the r.h.s of the last equation corresponds 
  to the variation between $t_1$ and $t_2$ of a non-negative local martingale, i.e., of a supermartingale.  
  It results that $\E[L_{t_2} - L_{t_1}| \cF_{t_1}] \leq 0$,
  which proves that $L_t$ is also a supermartingale.  
  Using this last inequality with $t_1 = 0$ and $t_2 = t$ gives
   $\E[L_{t}] \leq d$. This concludes the proof of
  Proposition~\ref{prop:new-counting-supermart}.
\end{proof}

\subsection{Bounding the odd powers of the dilation operator}

The process $\{ \bZ_t \}$ is not symmetric, hence following~\cite{tropp2012user}, we will force symmetry in our proofs by extending it in larger dimensions, using the symmetric dilation operator~\cite{paulsen2002completely} given, for a matrix $\bX$, by
\begin{equation}
\label{eq:sx}
\sS(\bs X) =
\begin{bmatrix}
\bs 0 & \bs X \\ \bs X^\top & \bs 0
\end{bmatrix}.
\end{equation}
The following Lemma will prove useful:

\begin{lemma}
\label{lem:boundSodd}
	Let $\bX$ be some $n \times m$ matrix and $k \in \N$.
	Then 
\begin{equation*}
 \sS(\bX)^{2k+1} =
 \begin{bmatrix}
 \bs 0 & \bX (\bX^\top \bX)^{k} \\
 \bX^\top (\bX \bX^\top)^{k} & 0
 \end{bmatrix} \mleq
 \begin{bmatrix}
 (\bX \bX^\top)^{k + 1/2} & \bs 0 \\
\bs 0 & (\bX^\top \bX)^{k + 1/2}
 \end{bmatrix}.
\end{equation*}
\end{lemma}

\begin{proof}

%\begin{equation}
% \sS(\bX)^{2k}  \! = 
%\begin{bmatrix}
% (\bX \bX ^\top)^k & \bs 0
% \\ \bs 0
% &  (\bX^\top \bX )^k
%\end{bmatrix} \! , \: \; \sS(\bX)^{2k+1} \! \! = \!
%\begin{bmatrix}
%\bs 0 & \bX (\bX^\top \bX)^{k} \\
%\bX^\top (\bX \bX^\top)^{k} & 0
%\end{bmatrix}.
%\end{equation}
The first equality results from a simple algebra. It can be rewritten as:
\begin{equation}
\label{eq:sodd}
\sS(\bX)^{2k+1} = \begin{bmatrix}
\bs 0 & (\bX \bX^\top)^{k} \bX \\
\bX^\top (\bX \bX^\top)^{k}  & 0
\end{bmatrix} = 
\bC 
\begin{bmatrix}
\bs 0 & (\bX \bX^\top)^{k}  \\
 (\bX \bX^\top)^{k}  & 0
\end{bmatrix}
\bC^\top
\end{equation}
where
\begin{equation}
\bC = 
\begin{bmatrix}
\bs 0 & \bI_n \\
\bX^\top  & 0
\end{bmatrix}.
\end{equation}
Since $(\bX \bX^\top)^{k} \mgeq \bs 0 $ and 
\begin{equation*}
  \bA = \begin{bmatrix}
  1 & -1 \\
  -1 &  1
  \end{bmatrix} \mgeq 0,
\end{equation*}
we obtain that $\bA \otimes (\bX \bX^\top)^{k} \mgeq \bs 0$, since the eigenvalues of a Kronecker product $\bA \otimes \bB$ are given by the  products of the eigenvalues of $\bA$ and $\bB$, see~\cite{golub2013matrix}. 
This leads to:
\begin{equation*}
  \begin{bmatrix}
  \bs 0 & (\bX \bX^\top)^{k}  \\
  (\bX \bX^\top)^{k}  & 0
  \end{bmatrix} \mleq
  \begin{bmatrix}
  (\bX \bX^\top)^{k} & \bs 0  \\
  \bs0 & (\bX \bX^\top)^{k}
  \end{bmatrix}.
\end{equation*}
Using the fact that~\cite{petz1994survey}
\begin{equation}
\label{eq:monotone_mult}
\bA \mleq \bB \;\; \Rightarrow \;\; \bC \bA \bC^\top \mleq \bC \bB
\bC^\top
\end{equation}
for any real matrices $\bA, \bB, \bC$ (with compatible dimensions),
we have:
$$
\sS(\bX)^{2k+1} \mleq \bC
\begin{bmatrix}
(\bX \bX^\top)^{k} & \bs 0  \\
\bs0 & (\bX \bX^\top)^{k} 
\end{bmatrix}
\bC^\top = 
\begin{bmatrix}
(\bX \bX^\top)^{k} & \bs 0  \\
\bs0 & (\bX^\top \bX)^{k+1} 
\end{bmatrix}.
$$
Along the same line, one can establish that:
$$
\sS(\bX)^{2k+1} \mleq 
\begin{bmatrix}
(\bX \bX^\top)^{k+1} & \bs 0  \\
\bs0 & (\bX^\top \bX)^{k} 
\end{bmatrix}.
$$
The square root of the product of the two inequalities provides the desired result.
\end{proof}

% section tools (end)

\section{Proof of Theorem~\ref{thm:concentration_counting}}
\label{sec:proof_of_theorem_1}

let us recall the definition~\eqref{eq:sx} of the dilation operator.
Let us point out that $\sS(\bX)$ is symmetric and satisfies $\lmax(\sS(\bX)) = \normop{\sS(\bX)} = \normop{\bX}$. 
Note that $\cS(\bZ_t)$ is purely discontinuous, so that
$\inr{\cS(\bZ)_{\bullet, j}^c}_t = \bO$ for any $j$. Recall that we work on events $\{ \lmax(\bV_t) \leq v \}$ and $\{ b_t \leq b\}$. 

%Let us recall that for a matrix $\bV$, we define $|\bV|$ by taking the absolute value of each entry of $\bV$.
We want to apply Proposition~\ref{prop:new-counting-supermart} (see Section~\ref{sec:tools} above) to $\bY_t = \xi \sS(\bZ_t) / b$.
In order to do so, we need the following Proposition.
\begin{proposition}
\label{prop:utmartingale}
Let the matrix $\bW_t$ be the matrix defined in 
Equation~\eqref{eq:W_def}.
Let any $\xi \ge 0$ be fixed and consider $\phi(x) = e^x - x - 1$ for 
$x \in \R$. Assume that
\begin{equation}
  \label{eq:ass_prop_utmartingale}
  \E \bigg[ \int_0^t
  \frac{\phi \big( \xi J_{\max} \norm{\bC_s}_\infty \max(\norm{\tT_s}_{\op; \infty},\norm{\tT_s^\top}_{\op; \infty})\big)}{J_{\max}^2 \norm{\bC_s}_\infty^2  \max(\norm{\tT_s}^2_{\op; \infty},\norm{\tT_s^\top}^2_{\op; \infty})} 
  (\bW_s)_{i, j} ds \bigg] < +\infty,
\end{equation}
for any $1 \leq i, j \leq m + n$ and grant also Assumption~\ref{ass:M_d} from Section~\ref{sec:discountinuous_martingale}. 
Then, the process
\begin{equation}
  \label{eq:U_t_def}
  \bU_t = \sum_{0 \le s \leq t} \Big( e^{\xi \Delta \sS(\bZ_s)} -
  \xi \Delta \sS(\bZ_{s}) - \bI \Big),
\end{equation}
admits a predictable, continuous and FV compensator $\bLambda_t$ given by
Equation~\eqref{eq:compensator_U_t_def} below.
Moreover, the following upper bound for the semi-definite order
\begin{equation}
  \label{eq:A_t_upper_bound}
  \bLambda_t \mleq \int_0^t \frac{\phi\big(\xi J_{\max} \norm{\bC_s}_\infty \max(\norm{\tT_s}_{\op; \infty},\norm{\tT_s^\top}_{\op; \infty})\big)}{J_{\max}^2 \norm{\bC_s}_\infty^2 \max(\norm{\tT_s}^2_{\op; \infty},\norm{\tT_s^\top}^2_{\op; \infty})}  \bW_s ds
\end{equation}
is satisfied for any $t > 0$.
\end{proposition}

This proposition is proved in Section~\ref{appendix:prop3} below.
We use Proposition~\ref{prop:new-counting-supermart}, 
Equation~\eqref{eq:A_t_upper_bound} and 
Equation~\eqref{eq:monotone_tr_exp} together with~\eqref{eq:ass_thm1} to obtain
\begin{equation*}
  \E \Big[\tr \exp \Big(\frac \xi b \sS(\bZ_t) - \int_0^t 
  \frac{\phi\big(\xi J_{\max} \norm{\bC_s}_\infty 
  \max(\norm{\tT_s}_{\op; \infty},\norm{\tT_s^\top}_{\op; \infty})  b^{-1} \big)}{J_{\max}^2 \norm{\bC_s}_\infty^2  \max(\norm{\tT_s}^2_{\op; \infty},\norm{\tT_s^\top}^2_{\op; \infty})}  \bW_s
  ds \Big) \Big] \leq m + n
\end{equation*}
for any $\xi \in [0, 3]$.
Using this with Lemma~\ref{lem:deviation} entails
  \begin{equation*}
    \P \bigg[ \frac{\lmax(\sS(\bZ_t))}{b} \geq \frac 1 \xi \lmax 
    \Big(\int_0^t \frac{\phi\big(\xi J_{\max} \norm{\bC_s}_\infty 
    \max(\norm{\tT_s}_{\op; \infty},\norm{\tT_s^\top}_{\op; \infty}) b^{-1}\big)}{J_{\max}^2 \norm{\bC_s}_\infty^2  \max(\norm{\tT_s}^2_{\op; \infty},\norm{\tT_s^\top}^2_{\op; \infty})}
    \bW_s ds \Big) + \frac x \xi \bigg] \leq (m + n) e^{-x}.
  \end{equation*}
  Note that on $\{ b_t \leq b \}$ we have $J_{\max} \norm{\bC_s}_\infty  \max(\norm{\tT_s}_{\op; \infty},\norm{\tT_s^\top}_{\op; \infty})b^{-1} \leq 1$ for any $s \in [0, t]$.
  The following facts on the function $\phi(x)$ hold true (cf.~\cite{massart2007concentration,hansen_reynaud_bouret_viroirard}):
  \begin{itemize}
  \item[(i)] $\phi(x h)\leq h^2 \phi(x)$ for any $h \in [0,1]$ and $x
    > 0$
  \item[(ii)] $\phi(\xi) \leq \frac{\xi^2}{2(1 - \xi / 3)}$ for any
    $\xi \in (0, 3)$
  \item[(iii)] $\min_{\xi \in (0, 1/c)} \big( \frac{a \xi}{1 - c \xi}
    + \frac x \xi \big) = 2 \sqrt{ax} + c x$ for any $a, c, x > 0$.
  \end{itemize}
  Using successively~(i) and~(ii), one gets, on $\{ b_t \leq b \} \cap \{
  \lmax(\bV_t) \leq v \}$, that for $\xi \in (0,3)$:
  \begin{align*}
    \frac 1 \xi \lmax\Big( &\int_0^t \frac{\phi\big(\xi J_{\max} 
    \norm{\bC_s}_\infty \max(\norm{\tT_s}_{\op; \infty},\norm{\tT_s^\top}_{\op; \infty})
    b^{-1}\big)}{J_{\max}^2 \norm{\bC_s}_\infty^2 
   \max(\norm{\tT_s}^2_{\op; \infty},\norm{\tT_s^\top}^2_{\op; \infty})}  \bW_s ds \Big)
    + \frac x \xi \\
    &\leq \frac{\phi(\xi)}{\xi b^2} \lmax\Big( \int_0^t
    \bW_s ds \Big) +
    \frac x \xi \\
    &= \frac{\phi(\xi)}{\xi b^2} \lmax(\bV_t) + \frac{x}{\xi} \\
    &\leq \frac{\xi v}{2 b^2 (1 - \xi / 3)} + \frac{x}{\xi},
  \end{align*}
  where we recall that $\bV_t$ is given
  by~\eqref{eq:def_V_t_counting}. This gives
  \begin{equation*}
    \P \bigg[ \frac{\lmax(\sS(\bZ_t))}{b} \geq  \frac{\xi v}{2 b^2 (1
      - \xi / 3)}  + \frac x \xi, \quad b_t \leq b, \quad \lmax(\bV_t)
    \leq v \bigg] \leq (m + n) e^{-x},
\end{equation*}
for any $ \xi \in (0,3)$. Now, by optimizing over $\xi$ using (iii)
(with $a = v/2b^2$ and $c = 1/3$), one obtains
\begin{equation*}
  \P \bigg[ \frac{\lmax(\sS(\bZ_t))}{b} \geq \frac{\sqrt{2 v x}}{b} +
  \frac x 3, \quad b_t \leq b, \quad \lmax(\bV_t) \leq v \bigg] \leq
  (m + n) e^{-x}.
\end{equation*}
Since $\lmax(\sS(\bZ_t)) = \normop{\sS(\bZ_t)}$, this concludes the
proof of Theorem~\ref{thm:concentration_counting} when the variance term is expressed using Equation~\eqref{eq:def_V_t_counting}.
It only remains to prove the fact that 
\begin{equation*}
    \sigma^2(\bZ_t) = \lmax(\bV_t).
\end{equation*}
Since $\bW_s$ is block-diagonal, we have obviously:
\begin{align*}
   \lmax(\bV_t) = \max \bigg( &\Big\| \int_0^t \tT_s \tT_s^\top \circ \big(\E(\bJ_1^{\odot 2}) 
   \odot \bC^{\odot 2}_s \odot \blambda_s \big) ds \Big\|_{\op}, \\
    & \quad \Big\| \int_0^t \tT_s^\top \tT_s \circ \big(\E(\bJ_1^{\odot 2}) \odot 
    \bC^{\odot 2}_s \odot \blambda_s \big) ds \Big\|_{\op} \bigg) \; .
\end{align*}
From the definition of $\bZ_t$,
since the entries of $\Delta \bM_t$ do not jump at the same time, the predictable quadratic covariation
of $(\bZ_t)_{k,j}$ and $(\bZ_t)_{l,j}$ is simply the predictable compensator of 
$\sum_{a,b} \sum_{s \leq t} (\tT_s)_{k,j;a,b} (\tT_s)_{l,j} (\bC_s)^2_{a,b} (\bJ_{N_s})^2_{a,b} (\Delta \bN_s)_{a,b}$. It results that
% we have $\E \left[(\Delta \bM_t)_{a,b} (\Delta \bM_t)_{a',b'} | \cF_{t-} \right] = dt \E((\bJ_1)_{a,b}^2) (\blambda_t)_{a,b}
% (\bC_t)_{a,b}^2 \delta_{a,a'} \delta_{b,b'}$. 
\begin{align*}
\sum_{j=1}^n (d \inr{\bZ_{\bul,j}}_t)_{k,l} = &\sum_j (\tT_t)_{k,j;a,b}(\tT_t)_{l,j;a,b} \E((\bJ_1)_{a,b}^2) (\blambda_t)_{a,b}
(\bC_t)_{a,b}^2 dt \\
 = & \left( \tT_t \tT_t^\top \circ \E(\bJ_1^{\odot 2}) \odot \bC_t^{\odot 2} \odot \blambda_t \right)_{k,l} dt.
\end{align*}
An analogous computation for $\inr{\bZ_{j,\bul}}_t$ leads to the expected result, and concludes the proof of Theorem~\ref{thm:concentration_counting}. $\hfill \qedhere$

\section{Proof of Proposition~\ref{prop:utmartingale}}
\label{appendix:prop3}

Let us first remark that:
$$
\exp(\sS(\bX)) = \sum_{k=0}^{\infty} \frac{1}{(2k) !}
\begin{bmatrix}
(\bX \bX ^\top)^k & \bs 0
\\ \bs 0 &  (\bX^\top \bX )^k
\end{bmatrix} + \frac{1}{(2k+1)!}
\begin{bmatrix}
\bs 0 & \bX (\bX^\top \bX)^{k} \\
\bX^\top (\bX \bX^\top)^{k} & 0
\end{bmatrix}.
$$
Then, from the definition of $\bU_t$ in Eq.~\eqref{eq:U_t_def}, we have:
\begin{align*}
  \bU_t  &= \sum_{0 \le s\leq t} \sum_{k \geq 2} \frac {\xi^k
    \sS(\Delta\bZ_s)^k}{k!} \\
    &= \sum_{0 \le s\leq t} \sum_{k \geq 1}
  \begin{bmatrix}
    \frac{\xi^{2k}}{(2k)!}  (\Delta \bZ_s \Delta \bZ_s^\top)^k &
    \frac{\xi^{2k+1}}{(2k+1)!} \Delta \bZ_s (\Delta \bZ_s^\top
    \Delta \bZ_s)^{k+1} \\
    \frac{\xi^{2k+1}}{(2k+1)!} \Delta \bZ_s^\top (\Delta \bZ_s \Delta
    \bZ_s^\top)^{k+1} & \frac{\xi^{2k}}{(2k)!}(\Delta \bZ_s^\top
    \Delta \bZ_s)^k
  \end{bmatrix}.
\end{align*}
Since $(\Delta\bZ_s (\Delta\bZ_s^\top \Delta\bZ_s)^k)^\top =
\Delta\bZ_s^\top(\Delta\bZ_s \Delta\bZ_s^\top)^k$, 
we need to compute three
terms: $(\Delta\bZ_s \Delta\bZ_s^\top)^k$,
$(\Delta\bZ_s^\top \Delta\bZ_s)^k$ and
$\Delta\bZ_s^\top(\Delta\bZ_s \Delta\bZ_s^\top)^k$. 

From Assumption~\ref{ass:M_d}, one has, a.s. that the entries of $\bM_t$
cannot jump at the same time, hence
\begin{equation}
  \label{eq:prod_delta_zero}
  \begin{split}
    (\Delta \bM_t)_{i_1, j_1}& \times \cdots \times (\Delta
    \bM_t)_{i_m, j_m} \\
    &=
    \begin{cases}
       ((\Delta \bM_t)_{i_1, j_1})^m &\text{ if } i_1 = \cdots 
       = i_m \text{ and } j_1 = \cdots = j_m \\
      \bO &\text{ otherwise}
    \end{cases}
  \end{split}
\end{equation}
a.s. for any $t$, $m \geq 2$ and any indexes $i_k \in \{ 1, \ldots, p \}$
and $j_k \in \{ 1, \ldots, q\}$. This entails, with the definition \eqref{eq:martingale} of $\Delta \bZ_s$, that 
$(\Delta \bZ_s \Delta \bZ_s^\top)^k$ is given, a.s., by
\begin{equation*}
    \sum_{a=1}^{p} \sum_{b=1}^{q} ( (\tT_s)_{\bul;a,b} 
    (\tT_s)_{\bul;a,b}^{\top})^k ((\bC_s)_{a,b} (\Delta \bM_s)_{a,b})^{2k} 
    = (\tT_s\tT_s^\top)^k \circ (\bC_s \odot \Delta \bM_s)^{\odot 2k}.
\end{equation*}
Let us remark that Equation~\eqref{eq:ass_prop_utmartingale} entails
\begin{equation*}
  \E \int_0^t  \sum_{k \geq 1} \frac{\xi^{2k}}{(2k)!} \sum_{a=1}^{p} 
  \sum_{b=1}^{q} \big(( (\tT_s)_{\bul;a,b} (\tT_s)_{\bul;a,b}^{\top})^k 
  \big)_{i,j} 
  ((\bC_s)_{a,b})^{2k} \; \E [|\bJ_1|_{a, b}^{2k}] \; (\blambda_s)_{a,b}
  \; ds < + \infty
\end{equation*}
for any $i, j$, so that together with Assumption~\ref{ass:M_d}, 
it is easily seen that the compensator of
\begin{equation}
  \label{eq:proof_compensator_first_term}
  \sum_{0 \le s\leq t} \sum_{k \geq 1}  \frac{\xi^{2k}}{(2k)!}  (\Delta \bZ_s \Delta \bZ_s^\top)^k
\end{equation}
is a.s. given by
\begin{equation*}
  \int_0^t  \sum_{k \geq 1} \frac{\xi^{2k}}{(2k)!} \sum_{a=1}^{p} 
  \sum_{b=1}^{q} ((\tT_s)_{\bul;a,b} (\tT_s)_{\bul;a,b}^{\top})^k 
  (\bC_s)_{a,b}^{2k}  \E[(\bJ_1)_{a,b}^{2k}] (\blambda_s)_{a, b}  ds.
\end{equation*}
Following the same arguments as 
for~\eqref{eq:proof_compensator_first_term}, we obtain that the 
compensator of 
\begin{equation*}
  \label{eq:proof_compensator_second_term_1}
  \sum_{0 \le s\leq t} \sum_{k \geq 1}  \frac{\xi^{2k}}{(2k)!}  (\Delta \bZ_s^\top \Delta \bZ_s)^k
\end{equation*}
is a.s. given by
\begin{equation*}
  \int_0^t  \sum_{k \geq 1} \frac{\xi^{2k}}{(2k)!} \sum_{a=1}^{p} \sum_{b=1}^{q} ((\tT_s)_{\bul;a,b}^{\top} (\tT_s)_{\bul;a,b} )^k (\bC_s)_{a,b}^{2k}  \E[(\bJ_1)_{a,b}^{2k}] (\blambda_s)_{a, b} ds.
\end{equation*}
Along the same line, one can easily show that the compensator of 
\begin{equation*}
  \label{eq:proof_compensator_second_term_2}
  \sum_{0 \le s\leq t} \sum_{k \geq 1}  \frac{\xi^{2k + 1}}{(2k + 1)!} \Delta \bZ_s^\top (\Delta \bZ_s \Delta \bZ_s^\top)^k,
\end{equation*}
reads a.s.:
\begin{equation*}
\int_0^t  \sum_{k \geq 1} \frac{\xi^{2k}}{(2k)!} \sum_{a=1}^{p} 
\sum_{b=1}^{q}  (\tT_s)_{\bul;a,b}^{\top} ((\tT_s)_{\bul;a,b} 
(\tT_s)_{\bul;a,b}^{\top} )^k (\bC_s)_{a,b}^{2k+1}  
\E[(\bJ_1)_{a,b}^{2k+1}] (\blambda_s)_{a, b} 
 ds.
\end{equation*}
Finally, we can write, a.s., the compensator of $\bU_t$ as
\begin{equation}
  \label{eq:compensator_U_t_def}
  \bLambda_t = \int_0^t  \sum_{k \geq 1} \bR_s^{(k)} ds
\end{equation}
where
\begin{equation}
  \label{eq:def_R_s}
  \bR_s^{(k)} =
  \begin{bmatrix}
    \frac{\xi^{2k}}{(2k)!} \bD_{1, s}^{(k)} &
    \frac{\xi^{2k+1}}{(2k+1)!} (\bH_{s}^{(k+1)})^\top \\
    \frac{\xi^{2k+1}}{(2k+1)!}  {\bH_{s}^{(k+1)}}
    &\frac{\xi^{2k}}{(2k)!}  \bD^{(k)}_{2, s}
  \end{bmatrix}
\end{equation}
with
\begin{align*}
  \bD_{1, s}^{(k)} &= \sum_{a=1}^{p} \sum_{b=1}^{q} ((\tT_s)_{\bul;a,b} 
  (\tT_s)_{\bul;a,b}^{\top})^k (\bC_s)_{a,b}^{2k} \E[(\bJ_1)_{a,b}^{2k}] 
  (\blambda_s)_{a, b} \\
  \bD_{2, s}^{(k)} &= \sum_{a=1}^{p} \sum_{b=1}^{q} 
  ((\tT_s)_{\bul;a,b}^\top 
  (\tT_s)_{\bul;a,b} )^k (\bC_s)_{a,b}^{2k} \E[(\bJ_1)_{a,b}^{2k}] 
  (\blambda_s)_{a, b} \\
  \bH_{s}^{(k+1)} &= \sum_{a=1}^{p} \sum_{b=1}^{q}  
  (\tT_s)_{\bul;a,b}^{\top} 
  ((\tT_s)_{\bul;a,b} (\tT_s)_{\bul;a,b}^{\top} )^k (\bC_s)_{a,b}^{2k+1} 
  \E[(\bJ_1)_{a,b}^{2k+1}] (\blambda_s)_{a, b}
\end{align*}
One can now directly use Lemma \ref{lem:boundSodd} with 
$\bX = (\tT_s)_{\bul;a,b} \E[(\bJ_1)_{a,b}^{2k+1}]^{1/(2k+1)} (\bC_s)_{a,b}$ to obtain:
\begin{align*}
\bLambda_t \mleq & \int_0^t \sum_{k \geq 2} \sum_{a=1}^p \sum_{b=1}^{q}
 \frac{\xi^k J_{\max}^{k-2}}{k!}
\begin{bmatrix}
  ((\tT_s)_{\bul;a,b} (\tT_s)_{\bul;a,b}^{\top} )^{k/2} & \bs 0 \\
  \bs 0 &  ((\tT_s)_{\bul;a,b}^\top (\tT_s)_{\bul;a,b})^{k/2}
\end{bmatrix} (\bC_s)_{a,b}^{k} \E[(\bJ_1)_{a,b}^2] (\blambda_s)_{a, b} 
ds \\
= & \int_0^t \sum_{k \geq 2}  \frac{\xi^k J_{\max}^{k-2}}{k!}
\begin{bmatrix}
(\tT_s \circ \tT_s^\top)^{k/2} & \! \! \! \! \bs 0 \\
\bs 0 & \! \! \! \! (\tT_s^{\top} \circ \tT_s)^{k/2}
\end{bmatrix} \circ \big( \bC_s^{\odot k}  \odot \E(\bJ_1^{\odot 2})
\odot \blambda_s\big) ds
\end{align*}
where we used the fact that $|(\bJ_1)_{i, j}| \leq J_{\max}$ a.s. for any $i, j$ under Assumption~\ref{ass:M_d}.
Given the fact that
\begin{equation*}
 ((\tT_s)_{\bul; a,b} (\tT_s)_{\bul;a,b}^\top)^{1/2} \mleq 
 \normop{(\tT_s)_{\bul; a,b}}  \bI_m 
 \mleq \norm{\tT_s}_{\op, \infty} \bI_m
\end{equation*}
for any $a, b$, where we used the notations and definitions from Section~\ref{sec:probback}, we have:
\begin{align*}
\bLambda_t \mleq & \int_0^t \begin{bmatrix}
\tT_s \tT_s^{\top}   & \bs 0 \\
\bs 0 &  \tT_s^\top \tT_s 
\end{bmatrix} 
\circ (\bC_s^{\odot 2} \odot \E(\bJ_1^{\odot 2}) \odot \blambda_s) 
 \sum_{k \geq 2} \frac{\xi^k}{k!} J_{\max}^{k-2} 
 \norm{\bC_s}_\infty^{k-2}  
 \max(\norm{\tT_s}_{\op; \infty},\norm{\tT_s^\top}_{\op; \infty})^{2k-1} ds \\
= & \int_0^t  
\begin{bmatrix}
\tT_s \tT_s^{\top}   & \bs 0 \\
\bs 0 &  \tT_s^\top \tT_s 
\end{bmatrix} 
\circ (\bC_s^{\odot 2} \odot \E(\bJ_1^{\odot 2}) \odot \blambda_s) \frac{\phi \big(\xi J_{\max} \norm{\bC_s}_\infty 
\max (\norm{\tT_s}_{\op; \infty},\norm{\tT_s^\top}_{\op; \infty})\big)}{J_{\max}^2 
\norm{\bC_s}_\infty^2 \max(\norm{\tT_s}^2_{\op; \infty},\norm{\tT_s^\top}^2_{\op; \infty})} ds
\end{align*}
where we recall that $\phi(x) = e^x - 1 - x$.
Hence, we finally get
\begin{equation*}
  \bLambda_t \mleq \int_0^t \frac{\phi \big(\xi J_{\max} \norm{\bC_s}_\infty 
  \max(\norm{\tT_s}_{\op; \infty},\norm{\tT_s^\top}_{\op; \infty})\big)}{J_{\max}^2 \norm{\bC_s}_\infty^2  \max(\norm{\tT_s}^2_{\op; \infty},\norm{\tT_s^\top}^2_{\op; \infty})} 
  \bW_s ds,
\end{equation*}
where $\bW_t$ is given by~\eqref{eq:W_def}. This concludes the proof
of Proposition~\ref{prop:utmartingale}. $\hfill \qed$

\section{Proof of Theorem~\ref{thm:concentration_continuous}}
\label{sec:proof_theorem2}

	The proof follows the same lines as the proof of
	Theorem~\ref{thm:concentration_counting}. We consider as before the
	symmetric dilation $\cS(\bZ_t)$ of $\bZ_t$ (see Eq. \eqref{eq:sx})  and apply
	Proposition~\ref{prop:new-counting-supermart} with $\bY_t = \xi
	\cS(\bZ_t)$ and $d = m + n$. Since $\bZ_t$ is a continuous
	martingale, we have $\bU_t = \bO$
	(cf.~\eqref{eq:ass_counting-supermart}), so that $\inr{\bU}_t = \bO$
	and we have $\inr{\bZ^c}_t = \inr{\bZ}_t$. So,
	Proposition~\ref{prop:new-counting-supermart} gives
	\begin{equation}
		\label{eq:majd}
		\E  \Big[ \tr \exp \Big( \xi \sS(\bZ_t) - \frac 1 2
		\sum_{j=1}^{m+n} \xi^2 \inr{\sS(\bZ)_{\bul,j}}_t \Big) \Big]
		\leq m + n.
	\end{equation}
	From the definition of the dilation operator $\sS$, it can be directly shown that:
	\begin{equation*}
		\sum_{j=1}^{m+n} \inr{\sS(\bZ)_{\bul,j}}_t  =
		\begin{bmatrix}
			\sum_{j=1}^n \inr{\bZ_{\bul,j}}_t & \bO_{m,n} \\ \bO_{n,m} &
			\sum_{j=1}^m \inr{\bZ_{j,\bul}}_t
		\end{bmatrix}
	\end{equation*}
	where $\inr{\bZ_{\bul,j}}_t$ (resp. $\inr{\bZ_{\bul,j}}_t$)
	is the $m \times m$ (resp. $n \times n$) matrix
	of the quadratic variation of the $j$-th column (resp. row) of $\bZ_t$.
	Since $[\bM^\con]_t = \inr{\bM^\con}_t = t \bI$, we have (for the
	sake of clarity, we omit the subscript $t$ in the matrices):
	\begin{align*}
		\sum_{j=1}^n (d\inr{\bZ_{\bul,j}}_t)_{kl} &=
		\sum_{j=1}^n  d[\bZ_{k,j},\bZ_{l,j}] \\
		&= \sum_{j=1}^n \sum_{a=1}^p \sum_{b=1}^q 
    \tT_{k,j;a,b} \tT_{l,j;a,b} \bC_{a,b}^2 dt \\		
		&= \big( \tT_t \tT_t^\top  \circ \bC_t^{\odot 2} \big)_{k,l} dt
	\end{align*}
	which gives in a matrix form
	\begin{equation*}
		\sum_{j=1}^n d\inr{\bZ_{\bul,j}}_t  =  \tT_t \tT_t^\top  \circ \bC_t^{\odot 2} dt
	\end{equation*}
	One can easily prove in the same way that
	\begin{equation*}
		\sum_{j=1}^m d\inr{\bZ_{j,\bul}}_t =  \tT_t^\top \tT_t \circ \bC_t^{\odot 2}  dt
	\end{equation*}
	Thus,
	$$
	\sum_{j=1}^{m+n} \inr{\sS(\bZ)_{\bul,j}^c}_t = \bV_t,
	$$
	where $\bV_t$ is given by \eqref{eq:def_V_t_continuous}.
	From \eqref{eq:majd}, it results
	\begin{eqnarray*}
		\E  \Big[ \tr \exp \Big(
		\xi\sS(\bZ_t)-\frac{\xi^2}{2} \bV_t
		\Big)  \Big] \leq m + n.
	\end{eqnarray*}
	Then, using Lemma \ref{lem:deviation}, one gets
	\begin{equation}
		\P \bigg[ \lmax(\sS(\bZ_t)) \geq \frac \xi 2 \lmax(\bV_t) + \frac
		x \xi \bigg] \leq (m+n) e^{-x}.
	\end{equation}
	On the event $\{\lmax( \bV_t ) \le v\}$, one gets
	\begin{equation}
		\P \bigg[ \lmax(\sS(\bZ_t)) \geq \frac \xi 2 v + \frac
		x \xi ~,~~~ \lmax( \bV_t ) \le v \bigg] \leq (m+n) e^{-x}.
	\end{equation}
	Optimizing on $\xi$, we apply this last result for
   $\xi = \sqrt{2x / v}$ and get
	\begin{equation}
		\P \bigg[ \lmax(\sS(\bZ_t)) \geq \sqrt{2xv}, \quad 
    \lmax( \bV_t ) \le v \bigg] \leq (m+n) e^{-x}.
	\end{equation}
	Since $ \lmax(\sS(\bZ_t)) = \normop{\sS(\bZ_t)}$,
	this concludes the proof of
  Theorem~\ref{thm:concentration_continuous}. $\hfill \qed$

\bibliographystyle{plain}

\bibliography{proba}

\end{document}